\documentclass[pdflatex,sn-mathphys-num]{sn-jnl}

\usepackage{graphicx}%
\usepackage{multirow}%
\usepackage{amsmath,amssymb,amsfonts}%
\usepackage{amsthm}%
\usepackage{mathrsfs}%
\usepackage[title]{appendix}%
\usepackage{xcolor}%
\usepackage{textcomp}%
\usepackage{manyfoot}%
\usepackage{booktabs}%
\usepackage{algorithm}%
\usepackage{algorithmicx}%
\usepackage{algpseudocode}%
\usepackage{listings}%

\theoremstyle{plain}
\newtheorem{theorem}{Theorem}[section]
\newtheorem{proposition}[theorem]{Proposition}
\newtheorem{lemma}[theorem]{Lemma}
\newtheorem{corollary}[theorem]{Corollary}

\theoremstyle{definition}
\newtheorem{definition}[theorem]{Definition}
\newtheorem{example}[theorem]{Example}

\theoremstyle{remark}
\newtheorem{remark}[theorem]{Remark}

\begin{document}

\title[Vector-Valued Wavelet Bases as Hilbert $\mathbb{M}_m(\mathbb{R})$-Module Bases: A Construction from Scalar Wavelets]{Vector-Valued Wavelet Bases as Hilbert $\mathbb{M}_m(\mathbb{R})$-Module Bases: A Construction from Scalar Wavelets}

\author*[1]{\fnm{Hicham} \sur{Tarif}}\email{hicham.tarif@uir.ac.ma}\equalcont{These authors contributed equally to this work.}

\author[1]{\fnm{Nadir} \sur{Maaroufi}}\email{nadir.maaroufi@uir.ac.ma}
\equalcont{These authors contributed equally to this work.}

\affil*[1]{\orgdiv{TICLab, College of Engineering and Architecture}, \orgname{International University of Rabat}, \orgaddress{ \city{Sala Al Jadida}, \postcode{11000}, \country{Morocco}}}


\abstract{Vector-valued multiscale representations are essential when signals or fields take values in $\mathbb{R}^m$ and component interactions carry meaningful information.
Most multiwavelet and super-wavelet constructions are formulated in scalar Hilbert-space settings and typically produce channelwise scalar coefficients followed by recombination.
We develop an intrinsic framework for vector-valued wavelets on $L^2(\mathbb{R}^d,\mathbb{R}^m)$ by endowing this space with a natural $\mathbb{M}_m(\mathbb{R})$-valued inner product, thereby turning it into a Hilbert $\mathbb{M}_m(\mathbb{R})$-module.
This module viewpoint yields matrix-valued coefficients that encode cross-component interactions and provides canonical reconstruction through a Parseval-type identity.
Within this setting, we introduce a constructive lifting procedure that builds separable multivariate vector-valued wavelet bases in $L^2(\mathbb{R}^d,\mathbb{R}^m)$ from scalar wavelet bases while preserving compact support, vanishing moments, and regularity.}




\maketitle


\section{Introduction}
\subsection{Preliminaries}\label{subsec:prelim}

Wavelets provide multiscale representations and are a standard tool in harmonic analysis and approximation theory.
Major advances in the 1980s, notably by Y.~Meyer and I.~Daubechies, led to compactly supported wavelet bases with prescribed regularity and vanishing moments, particularly efficient for approximating smooth functions.
Meyer's work also paved the way for Mallat's multiresolution analysis (MRA), which gives a systematic framework for constructing orthonormal wavelet bases in $L^2(\mathbb{R}^d)$ from the Hilbert space structure of $L^2(\mathbb{R})$ \cite{mallat1989multiresolution,mallat1999wavelet,daubechies1988orthonormal,lemarie1986ondelettes}.

The success of scalar wavelet bases motivates multiscale tools for vector-valued functions in $L^2(\mathbb{R}^d,\mathbb{R}^m)$. In many applications, vector-valued data are not merely independent channels: components are coupled and their interactions carry meaningful information.
This raises a structural question: \emph{what is the intrinsic functional-analytic framework in which vector-valued wavelet coefficients and reconstruction should be formulated?}
Componentwise processing yields only scalar coefficients, while a genuinely vector-valued transform should encode inter-component interactions and provide a canonical reconstruction acting on the full signal.

This paper addresses this question by developing a Hilbert-module viewpoint for $L^2(\mathbb{R}^d,\mathbb{R}^m)$, leading to matrix-valued coefficients and a multiresolution decomposition formulated at the module level.

\subsection{Related work and contributions}\label{subsec:related}

A natural starting point for vector-valued multiscale analysis is to view
$L^2(\mathbb{R}^d,\mathbb{R}^m)$ as the Hilbert space of $\mathbb{R}^m$-valued functions equipped with the scalar inner product
\[
\langle f,g\rangle=\int_{\mathbb{R}^d}\sum_{i=1}^{m} f_i(x)\,g_i(x)\,dx.
\]
Within this Hilbert-space viewpoint, two classical constructions have been widely used to process vector-valued data.

\paragraph{Multiwavelets}
Multiwavelets were introduced in the scalar space $L^2(\mathbb{R}^d)$ to gain flexibility in enforcing desirable properties
such as compact support, symmetry, orthogonality, and prescribed vanishing moments, by using several scaling functions and
associated wavelets \cite{chui1996study,strang1996wavelets}. They lead to matrix refinement equations already in the scalar
setting. When applied to vector-valued signals, a common implementation is componentwise (or block-diagonal): one applies a
scalar/multiwavelet transform to each channel separately and then recombines the outputs \cite{cotronei2002multiwavelet,fowler2002wavelet}.
This can provide perfect reconstruction, but the coefficient extraction is not intrinsically tied to the $\mathbb{R}^m$-geometry
of the signal. In particular, unless additional coupling mechanisms are introduced, cross-channel interactions are not encoded
at the level of the scalar pairing \cite{bhatt2007orthogonal,xia1996vector}.
\paragraph{Super-wavelets}
Super-wavelets provide wavelet bases in the direct-sum space
$L^2(\mathbb{R}^d)^{\oplus m}=L^2(\mathbb{R}^d)\oplus\cdots\oplus L^2(\mathbb{R}^d)$ by assembling several scalar wavelet
systems, first in Euclidean settings and later in broader contexts such as local fields
\cite{balan1999density,dutkay2005mra,shukla2018super,ahmad2022nonuniformsuper}. For vector-valued data, however, in their
standard direct-sum formulation these constructions have structural limitations: analysis is performed through scalar
Hilbert-space pairings on each copy of $L^2(\mathbb{R}^d)$, cross-channel interactions are not encoded intrinsically at the
coefficient level, and synthesis is obtained by assembling componentwise scalar expansions. Consequently, although
reconstruction is obtained within the direct-sum model, the analysis-synthesis mechanism is not intrinsic to the geometry
of $L^2(\mathbb{R}^d,\mathbb{R}^m)$. This differs from genuinely vector-valued transforms, where both coefficient extraction
and reconstruction are defined directly on $L^2(\mathbb{R}^d,\mathbb{R}^m)$ (see, e.g.,
\cite{bhatt2007orthogonal,xia1996vector,chen2007study,chen2008biorthogonal,chen2008biorthogonality}).

\paragraph{Hilbert-module and frame perspectives}
Hilbert $C^*$-module frames and related constructions are well established
\cite{FranckLarson2002,JakobsenLuef2020}. Their scope is largely structural: they are typically formulated at an abstract
level and do not, by themselves, identify a concrete intrinsic module model on $L^2(\mathbb{R}^d,\mathbb{R}^m)$ tailored to
vector-valued wavelet analysis, nor do they provide explicit wavelet-basis constructions that preserve classical properties
such as compact support and vanishing moments.

\paragraph{Main contributions} Our framework addresses these gaps by endowing $L^2(\mathbb{R}^d,\mathbb{R}^m)$ with a Hilbert $\mathbb{M}_m(\mathbb{R})$-module
structure and using the associated $\mathbb{M}_m(\mathbb{R})$-valued inner product to extract matrix-valued coefficients. This
makes the transform intrinsic to the vector geometry, encodes cross-component interactions through off-diagonal entries, and
yields canonical reconstruction via a Parseval-type identity (Corollary~\ref{mcr}). The
resulting coefficient architecture is equivariant under changes of basis in $\mathbb{R}^m$ (Lemma~\ref{lem:equivariance} and
Corollary~\ref{cor:invariants}. See also Examples~\ref{ex:geometry} and~\ref{ex:equivariance}).

\begin{itemize}
\item A structural module theorem (Theorem~\ref{thm:hilbert_onb_implies_parseval_module}) showing that Hilbert totality plus $\mathcal A$-orthonormality yields a Parseval $\mathcal A$-module frame, and therefore intrinsic reconstruction and energy identities for matrix-valued coefficients.
\item A constructive one-dimensional lifting theorem (Theorem~\ref{thm:vv-1d}) that converts scalar orthonormal wavelet bases into vector-valued wavelets in $L^2(\mathbb{R},\mathbb{R}^m)$ via dyadic scale grouping, while preserving compact support, regularity, and vanishing moments.
\item A multivariate extension theorem (Theorem~\ref{thm:main-vv-multi}) that builds separable vector-valued wavelet bases in $L^2(\mathbb{R}^d,\mathbb{R}^m)$ from scalar ingredients through tensorization and grouping.
\end{itemize}
\section{\texorpdfstring{$L^2(\mathbb{R}^d,\mathbb{R}^m)$}{L2(Rd,Rm)} as a Hilbert \texorpdfstring{$\mathbb{M}_m(\mathbb{R})$}{Mm(R)}-module}\label{subsec:module}

\subsection{Hilbert \texorpdfstring{$\mathbb{M}_m(\mathbb{R})$}{Mm(R)}-module and module-inner product}
Throughout the paper, $I_m$ and $0_m$ denote the identity and zero matrices in $\mathbb{M}_m(\mathbb{R})$, respectively.
We refer to \cite{lance1995hilbert,manuilov2000hilbert} for general background on Hilbert $C^*$-modules.

In this work the coefficient algebra is the finite-dimensional $*$-algebra
$\mathcal{A}=\mathbb{M}_m(\mathbb{R})$, endowed with the involution $A^*=A^\top$.
We write $\mathrm{tr}$ for the trace.
When needed, we equip $\mathbb{M}_m(\mathbb{R})$ with the Frobenius norm
$$
\|A\|_{F}:=\Big(\sum_{i,j=1}^m |a_{ij}|^2\Big)^{1/2},\qquad A=(a_{ij})\in\mathbb{M}_m(\mathbb{R}),
$$
although, since $\mathbb{M}_m(\mathbb{R})$ is finite-dimensional, all matrix norms are equivalent and the specific choice is immaterial for completeness considerations.

\begin{definition}[Hilbert module {\cite[Chapter~1]{lance1995hilbert}}]\label{def:HM}
Let $\mathcal{A}$ be a (real) $C^*$-algebra.
A \emph{pre-Hilbert $\mathcal{A}$-module} is a left $\mathcal{A}$-module $\mathcal{H}$ equipped with an $\mathcal{A}$-valued inner product $\langle\cdot,\cdot\rangle_{\mathcal{A}}:\mathcal{H}\times\mathcal{H}\to\mathcal{A}$ such that for all $x,y\in\mathcal{H}$ and $A\in\mathcal{A}$:
\begin{enumerate}[label=(\roman*)]

  \item $\langle x,x\rangle_{\mathcal{A}}\ge 0$;
  \item $\langle x,x\rangle_{\mathcal{A}}=0$ implies $x=0$;
  \item $\langle x,y\rangle_{\mathcal{A}}=\langle y,x\rangle_{\mathcal{A}}^{*}$;
  \item $\langle Ax,y\rangle_{\mathcal{A}}=A\,\langle x,y\rangle_{\mathcal{A}}$.

\end{enumerate}
The induced norm is $\|x\|_{\mathcal{A}}:=\|\langle x,x\rangle_{\mathcal{A}}\|_{\mathcal{A}}^{1/2}$.
If $\mathcal{H}$ is complete for this norm, it is called a \emph{Hilbert $\mathcal{A}$-module}.
\end{definition}

We now endow $L^2(\mathbb{R}^d,\mathbb{R}^m)$ with an $\mathbb{M}_m(\mathbb{R})$-valued inner product.
Let $\{e_i\}_{i=1}^m$ be the canonical basis of $\mathbb{R}^m$ and write $f=(f_1,\dots,f_m)^\top$.
Define $\langle f,g\rangle_{\mathbb{M}_m(\mathbb{R})}\in\mathbb{M}_m(\mathbb{R})$ by
\begin{equation}\label{eq:mm-innerproduct}
\big(\langle f,g\rangle_{\mathbb{M}_m(\mathbb{R})}\big)_{ij}
:=\int_{\mathbb{R}^d} f_i(x)\,g_j(x)\,dx,
\qquad f,g\in L^2(\mathbb{R}^d,\mathbb{R}^m).
\end{equation}
Equivalently, $\langle f,g\rangle_{\mathbb{M}_m(\mathbb{R})}=\int_{\mathbb{R}^d} f(x)\,g(x)^\top\,dx$.

\begin{proposition}\label{prop:HM-module}
The pair $(L^2(\mathbb{R}^d,\mathbb{R}^m),\langle\cdot,\cdot\rangle_{\mathbb{M}_m(\mathbb{R})})$ is a Hilbert module over $\mathbb{M}_m(\mathbb{R})$.
\end{proposition}

\begin{proof}
We verify Definition~\ref{def:HM}.

\emph{(iv) Left $\mathbb{M}_m(\mathbb{R})$-linearity.}
For $A\in\mathbb{M}_m(\mathbb{R})$ and $f,g\in L^2(\mathbb{R}^d,\mathbb{R}^m)$, using \eqref{eq:mm-innerproduct} and finite-dimensionality,
\[
\big(\langle Af,g\rangle_{\mathbb{M}_m(\mathbb{R})}\big)_{ij}
=\int_{\mathbb{R}^d} (Af)_i(x)\,g_j(x)\,dx
=\sum_{k=1}^m A_{ik}\int_{\mathbb{R}^d} f_k(x)\,g_j(x)\,dx
=\big(A\langle f,g\rangle_{\mathbb{M}_m(\mathbb{R})}\big)_{ij}.
\]

\emph{(iii) Symmetry.}
By definition, $\big(\langle f,g\rangle_{\mathbb{M}_m(\mathbb{R})}\big)_{ij}
=\int_{\mathbb{R}^d} f_i(x)\,g_j(x)\,dx
=\big(\langle g,f\rangle_{\mathbb{M}_m(\mathbb{R})}\big)_{ji}$,

hence $\langle f,g\rangle_{\mathbb{M}_m(\mathbb{R})}=\langle g,f\rangle_{\mathbb{M}_m(\mathbb{R})}^{\top}$.

\emph{(i) Positivity.}
For any $v\in\mathbb{R}^m$,
\[
v^\top\langle f,f\rangle_{\mathbb{M}_m(\mathbb{R})}v
=\int_{\mathbb{R}^d} (v^\top f(x))^2\,dx \ge 0,
\]
so $\langle f,f\rangle_{\mathbb{M}_m(\mathbb{R})}$ is positive semidefinite.

\emph{(ii) Definiteness.}
If $\langle f,f\rangle_{\mathbb{M}_m(\mathbb{R})}=0$, then taking $v=e_i$ yields
$\int_{\mathbb{R}^d} |f_i(x)|^2\,dx=0$ for each $i$, hence $f=0$.

Therefore $\langle\cdot,\cdot\rangle_{\mathbb{M}_m(\mathbb{R})}$ is an $\mathbb{M}_m(\mathbb{R})$-valued inner product.

Finally, we note that $\langle f,g\rangle_{\mathbb{M}_m(\mathbb{R})}$ is well-defined: for each $i,j$,
by Cauchy--Schwarz,
\[
\left|\big(\langle f,g\rangle_{\mathbb{M}_m(\mathbb{R})}\big)_{ij}\right|
=\left|\int_{\mathbb{R}^d} f_i(x)\,g_j(x)\,dx\right|
\le \|f_i\|_{L^2}\,\|g_j\|_{L^2}<\infty.
\]

Moreover,
\[
\mathrm{tr}\big(\langle f,f\rangle_{\mathbb{M}_m(\mathbb{R})}\big)
=\sum_{i=1}^m \int_{\mathbb{R}^d} |f_i(x)|^2\,dx
=\|f\|_{L^2}^2.\]

Since $\mathbb{M}_m(\mathbb{R})$ is finite-dimensional, all matrix norms on it are equivalent; in particular, fixing any matrix norm $\|\cdot\|$,
there exist constants $c,C>0$ such that for every positive semidefinite $A\in\mathbb{M}_m(\mathbb{R})$,
\[
c\,\mathrm{tr}(A)\le \|A\| \le C\,\mathrm{tr}(A).
\]
Applying this to $A=\langle f,f\rangle_{\mathbb{M}_m(\mathbb{R})}$ yields
\[
c\,\|f\|_{L^2}^2 \le \|\langle f,f\rangle_{\mathbb{M}_m(\mathbb{R})}\|
\le C\,\|f\|_{L^2}^2,
\]
hence
\[
\sqrt{c}\,\|f\|_{L^2}\le \|f\|_{\mathbb{M}_m(\mathbb{R})}\le \sqrt{C}\,\|f\|_{L^2}.
\]
Therefore the module norm $\|\cdot\|_{\mathbb{M}_m(\mathbb{R})}$ is equivalent to $\|\cdot\|_{L^2}$, and
$(L^2(\mathbb{R}^d,\mathbb{R}^m),\|\cdot\|_{\mathbb{M}_m(\mathbb{R})})$ is complete since it is the same
underlying space endowed with an equivalent norm. This proves that it is a Hilbert
$\mathbb{M}_m(\mathbb{R})$-module.

\end{proof}

\subsubsection{Vector geometry preservation: covariance under channel mixing}

The following lemma records the natural covariance of the $\mathbb{M}_m(\mathbb{R})$-valued inner product under constant channel mixing. In particular, restricting to $U\in O(m)=\{U\in\mathbb{M}_m(\mathbb{R})\;:\;U^\top U=I_m\}$ yields the usual equivariance under orthogonal channel mixing (with preservation of scalar $L^2$-energy).

\begin{lemma}[Covariance under constant channel mixing]\label{lem:equivariance}
Let $U\in \mathbb{M}_m(\mathbb{R})$ be constant and define $(Uf)(x):=U\,f(x)$ for $f\in L^2(\mathbb{R}^d,\mathbb{R}^m)$.
Then for all $f,g\in L^2(\mathbb{R}^d,\mathbb{R}^m)$,

\[
\langle Uf,Ug\rangle_{\mathbb{M}_m(\mathbb{R})}
=U\,\langle f,g\rangle_{\mathbb{M}_m(\mathbb{R})}\,U^{\top}.
\]
\end{lemma}

\begin{proof}
This is an immediate consequence of properties (iii) and (iv) of the
$\mathbb{M}_m(\mathbb{R})$-valued inner product:
\[
\langle Uf,Ug\rangle_{\mathbb{M}_m(\mathbb{R})}
=U\,\langle f,Ug\rangle_{\mathbb{M}_m(\mathbb{R})}
=U\,\langle Ug,f\rangle_{\mathbb{M}_m(\mathbb{R})}^{\top}
=U\,\big(U\,\langle g,f\rangle_{\mathbb{M}_m(\mathbb{R})}\big)^{\top}
=U\,\langle f,g\rangle_{\mathbb{M}_m(\mathbb{R})}\,U^{\top}.
\]
\end{proof}
\begin{example}[Same scalar energy, different vector geometry]\label{ex:geometry}
Let $m=2$ and pick $u\in L^2(\mathbb{R}^d)$, $u\neq 0$. Define
 $f:=(u,0)^{\top}$ and $g:=\frac{1}{\sqrt{2}}(u,u)^{\top}$.Then $\|f\|_{L^2}^2=\|g\|_{L^2}^2=\|u\|_2^2$, so a purely scalar-energy viewpoint does not distinguish them. However, their $\mathbb{M}_2(\mathbb{R})$-valued ``energies'' differ:
\[
\langle f,f\rangle_{\mathbb{M}_2(\mathbb{R})}=
\begin{pmatrix}\|u\|_2^2&0\\0&0\end{pmatrix},
\qquad
\langle g,g\rangle_{\mathbb{M}_2(\mathbb{R})}=\frac{\|u\|_2^2}{2}\begin{pmatrix}1&1\\1&1\end{pmatrix}.
\]
Thus $g$ exhibits nontrivial off-diagonal terms already at the level of $\langle g,g\rangle_{\mathbb{M}_2(\mathbb{R})}$; these encode cross-channel correlations and are invisible to componentwise scalar pairings (see Remark~\ref{rem:vector-geometry}).
\end{example}

\begin{example}[Equivariance and basis-independent anisotropy]\label{ex:equivariance}
Let $m=2$ and let $u\in L^2(\mathbb{R}^d)$ with $u\neq 0$. Define $f:=(u,0)^{\top}$ and, for $\theta\in\mathbb{R}$,
let $U_\theta\in O(2)$ be the rotation matrix
\[
U_\theta:=
\begin{pmatrix}
\cos\theta&-\sin\theta\\
\sin\theta&\cos\theta
\end{pmatrix}.
\]
Set $f_\theta:=U_\theta f$. Then $\|f_\theta\|_{L^2}=\|f\|_{L^2}$ for all $\theta$, while
\[
\langle f_\theta,f_\theta\rangle_{\mathbb{M}_2(\mathbb{R})}
=U_\theta\,\langle f,f\rangle_{\mathbb{M}_2(\mathbb{R})}\,U_\theta^{\top}.
\]
In particular, the spectrum of $\langle f_\theta,f_\theta\rangle_{\mathbb{M}_2(\mathbb{R})}$ is independent of $\theta$ and equals $\{\|u\|_2^2,0\}$, providing a basis-independent measure of anisotropy that is invisible to the scalar energy $\|f_\theta\|_{L^2}^2$.
\end{example}

\begin{remark}[Vector geometry preservation (in a precise sense)]

The ``geometry'' preserved by our construction is the one induced by the $\mathbb{M}_m(\mathbb{R})$-valued inner product

\[
\big(\langle f,g\rangle_{\mathbb{M}_m(\mathbb{R})}\big)_{ij}
=\int_{\mathbb{R}^d} f_i(x)\,g_j(x)\,dx,
\]

\noindent which is the matrix of pairwise $L^2$ pairings between components. In particular, off-diagonal entries encode
cross-component interactions that are invisible to purely componentwise scalar coefficient extraction. In particular, the off-diagonal entries are not recoverable from componentwise scalar coefficients without extra coupling assumptions.
Moreover, the framework is equivariant under changes of basis in $\mathbb{R}^m$: for any $U\in O(m), \langle Uf,Ug\rangle_{\mathbb{M}_m(\mathbb{R})}=U\,\langle f,g\rangle_{\mathbb{M}_m(\mathbb{R})}\,U^{\top}$ (Lemma~\ref{lem:equivariance}). Therefore, canonical invariants of $\langle f,f\rangle_{\mathbb{M}_m(\mathbb{R})}$ (e.g.\ trace, spectrum) are basis-independent. This provides a precise meaning in which the vector structure is respected.
\end{remark}

\begin{corollary}[Basis-independent geometric quantities]\label{cor:invariants}
Let $U\in \mathbb{M}_m(\mathbb{R})$ be constant and $f\in L^2(\mathbb{R}^d,\mathbb{R}^m)$. Then

\[
\langle Uf,Uf\rangle_{\mathbb{M}_m(\mathbb{R})}
=U\,\langle f,f\rangle_{\mathbb{M}_m(\mathbb{R})}\,U^{\top}.
\]

If $U\in O(m)$, then $\langle f,f\rangle_{\mathbb{M}_m(\mathbb{R})}$ and $\langle Uf,Uf\rangle_{\mathbb{M}_m(\mathbb{R})}$ have the same spectrum, and hence
\[
\mathrm{tr}\big(\langle Uf,Uf\rangle_{\mathbb{M}_m(\mathbb{R})}\big)
=\mathrm{tr}\big(\langle f,f\rangle_{\mathbb{M}_m(\mathbb{R})}\big).
\]
\end{corollary}

\begin{proof}
The identity follows from Lemma~\ref{lem:equivariance} by taking $g=f$.
Assume now $U\in O(m)$, so $U^\top=U^{-1}$. Then the matrices $A=\langle f,f\rangle_{\mathbb{M}_m(\mathbb{R})}$ and
$UAU^\top$ are similar (indeed $UAU^\top = U A U^{-1}$), hence they have the same spectrum.
The trace identity follows from $\mathrm{tr}(UAU^\top)=\mathrm{tr}(A U^\top U)=\mathrm{tr}(A)$.
\end{proof}

\subsection{Key properties of the \texorpdfstring{$\mathbb{M}_m(\mathbb{R})$}{Mm(R)}-inner product}

We collect a few properties of the $\mathbb{M}_m(\mathbb{R})$-valued inner product
\[
\big(\langle f,g\rangle_{\mathbb{M}_m(\mathbb{R})}\big)_{ij}
=\int_{\mathbb{R}^d} f_i(x)\,g_j(x)\,dx,
\qquad f,g\in L^2(\mathbb{R}^d,\mathbb{R}^m).
\]
All results hold for general $d\ge 1$ and $m\ge 1$.

\begin{proposition}[Continuity / Cauchy--Schwarz]\label{P1}
Fix $f\in L^2(\mathbb{R}^d,\mathbb{R}^m)$. The map
\[
T_f:\ L^2(\mathbb{R}^d,\mathbb{R}^m)\to \mathbb{M}_m(\mathbb{R}),\qquad
T_f(g):=\langle f,g\rangle_{\mathbb{M}_m(\mathbb{R})},
\]
is linear and bounded. More precisely, for any matrix norm $\|\cdot\|$ on $\mathbb{M}_m(\mathbb{R})$,
there exists a constant $C_m>0$ (depending only on $m$ and on the chosen norm) such that
\[
\|\langle f,g\rangle_{\mathbb{M}_m(\mathbb{R})}\|
\le C_m\,\|f\|_{L^2(\mathbb{R}^d,\mathbb{R}^m)}\,\|g\|_{L^2(\mathbb{R}^d,\mathbb{R}^m)}
\qquad\forall\,g\in L^2(\mathbb{R}^d,\mathbb{R}^m).
\]
In particular, $T_f$ is continuous.
\end{proposition}

\begin{proof}
Linearity is immediate from the definition.
For each pair $(i,j)$, the scalar Cauchy--Schwarz inequality gives
\[
\big|\big(\langle f,g\rangle_{\mathbb{M}_m(\mathbb{R})}\big)_{ij}\big|
=\left|\int_{\mathbb{R}^d} f_i(x)\,g_j(x)\,dx\right|
\le \|f_i\|_{L^2(\mathbb{R}^d)}\,\|g_j\|_{L^2(\mathbb{R}^d)}
\le \|f\|_{L^2}\,\|g\|_{L^2}.
\]
Hence all entries are bounded by $\|f\|_{L^2}\|g\|_{L^2}$.
Since $\mathbb{M}_m(\mathbb{R})$ is finite-dimensional, any matrix norm $\|\cdot\|$
is dominated by a constant multiple of the max-entry norm. Therefore there exists $C_m>0$ such that
\[
\|\langle f,g\rangle_{\mathbb{M}_m(\mathbb{R})}\|
\le C_m\,\|f\|_{L^2}\,\|g\|_{L^2}.
\]
\end{proof}
\begin{proposition}[Compatibility with separable tensor products]\label{prop:hadamard}
Let $d_1,d_2\ge 1$ and $m\ge 1$.
Let $f,r\in L^2(\mathbb{R}^{d_1},\mathbb{R}^m)$ and $g,u\in L^2(\mathbb{R}^{d_2},\mathbb{R}^m)$.
Define $h,k\in L^2(\mathbb{R}^{d_1+d_2},\mathbb{R}^m)$ by
\[
h(x,y):= f(x)\odot g(y),\qquad
k(x,y):= r(x)\odot u(y),
\]
where $\odot$ denotes the componentwise product: $(a\odot b)_i:=a_i b_i$.
This compatibility will be used later in our separable construction of multivariate vector-valued wavelets.
Then

\[
\langle h,k\rangle_{\mathbb{M}_m(\mathbb{R})}
=
\langle f,r\rangle_{\mathbb{M}_m(\mathbb{R})}\ \odot\ \langle g,u\rangle_{\mathbb{M}_m(\mathbb{R})},
\]
where on the right-hand side $\odot$ denotes the Hadamard (entrywise) product of matrices:
$(A\odot B)_{ij}:=A_{ij}B_{ij}$.
\end{proposition}

\begin{proof}
Fix $i,j\in\{1,\dots,m\}$. By definition and Fubini,
\begin{align*}
\big(\langle h,k\rangle_{\mathbb{M}_m(\mathbb{R})}\big)_{ij}
&=\int_{\mathbb{R}^{d_1+d_2}} h_i(x,y)\,k_j(x,y)\,dx\,dy\\
&=\int_{\mathbb{R}^{d_1+d_2}} f_i(x)\,g_i(y)\,r_j(x)\,u_j(y)\,dx\,dy\\
&=\left(\int_{\mathbb{R}^{d_1}} f_i(x)\,r_j(x)\,dx\right)
  \left(\int_{\mathbb{R}^{d_2}} g_i(y)\,u_j(y)\,dy\right)\\
&=\big(\langle f,r\rangle_{\mathbb{M}_m(\mathbb{R})}\big)_{ij}\,
  \big(\langle g,u\rangle_{\mathbb{M}_m(\mathbb{R})}\big)_{ij}.
\end{align*}
This is exactly the $(i,j)$-entry of
$\langle f,r\rangle_{\mathbb{M}_m(\mathbb{R})}\odot \langle g,u\rangle_{\mathbb{M}_m(\mathbb{R})}$.
\end{proof}

\begin{remark}[Why this encodes vector geometry]\label{rem:vector-geometry}
Since for all $f\in L^2(\mathbb{R}^d,\mathbb{R}^m)$ one has
\[
\mathrm{tr}\big(\langle f,f\rangle_{\mathbb{M}_m(\mathbb{R})}\big)
=\sum_{i=1}^m \int_{\mathbb{R}^d} |f_i(x)|^2\,dx
=\|f\|_{L^2(\mathbb{R}^d,\mathbb{R}^m)}^2.
\]
Then the trace of $\langle f,f\rangle_{\mathbb{M}_m(\mathbb{R})}$  recovers the classical $L^2$ energy. It follows that  the matrix $\langle f,f\rangle_{\mathbb{M}_m(\mathbb{R})}$ contains, besides the diagonal energies
$\int_{\mathbb{R}^d} |f_i(x)|^2\,dx$, the off-diagonal cross terms $\int_{\mathbb{R}^d} f_i(x)\,f_j(x)\,dx$ ($i\neq j$), which quantify non-normalized cross-channel correlations.
This illustrates how the $\mathbb{M}_m(\mathbb{R})$-valued inner product retains cross-channel information that is invisible to scalar energy considerations.

\end{remark}

\subsection{Module frames, bases, and intrinsic reconstruction}\label{subsec:parseval}
The Hilbert $\mathbb{M}_m(\mathbb{R})$-module structure provides a natural coefficient algebra and a canonical reconstruction mechanism for vector-valued wavelet analysis. In particular, it allows one to define orthonormal module bases and Parseval module frames whose coefficients are intrinsically matrix-valued. This subsection recalls the VMRA setting and records the basic consequences of working with the $\mathbb{M}_m(\mathbb{R})$-valued inner product. We begin with the countable generation property.

\begin{proposition}\label{prop:countgen}
The Hilbert module $(L^2(\mathbb{R}^d,\mathbb{R}^m),\langle\cdot,\cdot\rangle_{\mathbb{M}_m(\mathbb{R})})$
is countably generated. Moreover, every closed submodule admits an orthogonal complement.
\end{proposition}

\begin{proof}
Since $L^2(\mathbb{R}^d,\mathbb{R}^m)$ is separable, the Hilbert module
$(L^2(\mathbb{R}^d,\mathbb{R}^m),\langle\cdot,\cdot\rangle_{\mathbb{M}_m(\mathbb{R})})$ is countably generated; see
\cite[Chapter~6]{lance1995hilbert}.
Moreover, because $\mathbb{M}_m(\mathbb{R})$ is finite-dimensional, it is (as a $C^*$-algebra) isomorphic to the algebra of compact
operators on a finite-dimensional Hilbert space, and hence every closed submodule of a countably generated Hilbert
$\mathbb{M}_m(\mathbb{R})$-module is orthogonally complemented (see, e.g., \cite[Chapter~15]{manuilov2000hilbert}).
\end{proof}

We emphasize that the coefficient structure used in this paper is not postulated through matrix refinement masks; it arises canonically from the $\mathbb{M}_m(\mathbb{R})$-valued inner product defining the Hilbert-module structure on $L^2(\mathbb{R}^d,\mathbb{R}^m)$. We record this via module orthonormal bases and Parseval module frames.

\begin{definition}[Orthonormal systems, Parseval frames, and standard bases in Hilbert $\mathcal{A}$-modules]\label{def:module-frames-bases}
Let $\mathbb{H}$ be a Hilbert $\mathcal{A}$-module.
\begin{enumerate}
\item A family $S=\{h_i\}_{i\in I}\subset \mathbb{H}$ is an \emph{orthonormal system (ONS)} if
\begin{equation}\label{eq:AONS}
\langle h_i,h_j\rangle_{\mathcal{A}}=0 \ (i\neq j),\qquad
\langle h_i,h_i\rangle_{\mathcal{A}}=p_i,
\end{equation}
where each $p_i\in\mathcal{A}$ is a nonzero projection ($p_i^2=p_i=p_i^*$). If moreover $p_i=1_{\mathcal{A}}$ for all $i$, then $S$ is called an \emph{orthonormal system of unit vectors}.
\item A family $\{x_i\}_{i\in I}\subset \mathbb{H}$ is a \emph{Parseval $\mathcal{A}$-module frame} if for every $f\in\mathbb{H}$,
\[
f=\sum_{i\in I}\langle f,x_i\rangle_{\mathcal{A}}\,x_i
\]
with convergence in the Hilbert-module norm.
\item A family $S=\{h_i\}_{i\in I}\subset \mathbb{H}$ is an \emph{$\mathcal{A}$-orthonormal basis} (or \emph{standard basis}) if it is an orthonormal system of unit vectors and for every $f\in\mathbb{H}$ one has the reconstruction formula
\begin{equation}\label{eq:AONB-reconstruction}
f=\sum_{i\in I}\langle f,h_i\rangle_{\mathcal{A}}\,h_i,
\end{equation}
with convergence in the Hilbert-module norm. Equivalently, $S$ is an $\mathcal{A}$-Parseval frame and is orthonormal.
\end{enumerate}
\end{definition}

\begin{definition}[$\mathbb{M}_m(\mathbb{R})$-orthonormal system of translates]\label{def:mm-ons-translates}
Let $\mathcal{A}=\mathbb{M}_m(\mathbb{R})$ and $\mathbb{H}=L^2(\mathbb{R}^d,\mathbb{R}^m)$
endowed with the $\mathcal A$-valued inner product
$\langle f,g\rangle_{\mathcal{A}}=\int_{\mathbb{R}^d} f(x)\,g(x)^\top\,dx$.
A family $\{U_k\}_{k\in\mathbb Z^d}\subset\mathbb H$ is called
\emph{$\mathcal A$-orthonormal (of unit vectors)} if
\[\langle U_k,U_u\rangle_{\mathcal A}=\delta_{k,u}\,I_m,
\qquad k,u\in\mathbb Z^d.
\]
If $U\in\mathbb H$ and $\{U(\cdot-k)\}_{k\in\mathbb Z^d}$ is $\mathcal A$-orthonormal,
we call $U$ an \emph{$\mathcal A$-orthonormal vector-valued function}.
\end{definition}

\begin{proposition}[Uniqueness of coefficients for orthonormal unit systems]\label{prop:coeff-unique}
Let $S=\{h_i\}_{i\in I}\subset\mathbb H$ be $\mathcal A$-orthonormal with
$\langle h_i,h_i\rangle_{\mathcal A}=1_{\mathcal A}$, and assume the inner product is
continuous with respect to norm convergence.
Assume $f=\sum_{i\in I} a_i h_i$ with convergence in norm, where $(a_i)\subset\mathcal A$.
Then necessarily $a_i=\langle f,h_i\rangle_{\mathcal A}$ for all $i$.
\end{proposition}

\begin{proof}
Fix $j\in I$. By the continuity assumption in the statement (which is satisfied in our concrete setting by Proposition~2.8) and $\mathcal A$-linearity in the first variable,
\[
\langle f,h_j\rangle_{\mathcal A}
=\Big\langle \sum_{i\in I} a_i h_i,\ h_j\Big\rangle_{\mathcal A}
=\sum_{i\in I} a_i\,\langle h_i,h_j\rangle_{\mathcal A}
=a_j\,1_{\mathcal A}=a_j.
\]
\end{proof}

\begin{remark}[About $\ell^2(\mathbb M_m(\mathbb R))$]\label{rem:l2A}
In the finite-dimensional case $\mathcal A=\mathbb M_m(\mathbb R)$, any two norms
(on matrices and on coefficient sequences) are equivalent.
Thus, after establishing \eqref{eq:AONB-reconstruction}, one can deduce that
the coefficient family $\{\langle f,h_i\rangle_{\mathcal A}\}_{i\in I}$ belongs to
$\ell^2(\mathcal A)$ with respect to (for instance) the Frobenius norm.

\end{remark}

\begin{corollary}[Parseval-type identity and intrinsic reconstruction]\label{prop:parseval}\label{mcr}
Assume that the family $\{\Psi_{j,k}\}_{(j,k)\in\mathbb{Z}\times\mathbb{Z}^d}\subset L^2(\mathbb{R}^d,\mathbb{R}^m)$
is a Parseval $\mathbb{M}_m(\mathbb{R})$-module frame, i.e.\ for every $f\in L^2(\mathbb{R}^d,\mathbb{R}^m)$ one has the reconstruction
\begin{equation}\label{eq:parseval-recons}
f=\sum_{(j,k)\in\mathbb{Z}\times\mathbb{Z}^d}\langle f,\Psi_{j,k}\rangle_{\mathbb{M}_m(\mathbb{R})}\,\Psi_{j,k},
\end{equation}
with convergence in the Hilbert-module norm. For $f\in L^2(\mathbb{R}^d,\mathbb{R}^m)$ define
\[
C_{j,k}:=\langle f,\Psi_{j,k}\rangle_{\mathbb{M}_m(\mathbb{R})}\in\mathbb{M}_m(\mathbb{R}).
\]
Then:
\[
\sum_{(j,k)\in\mathbb{Z}\times\mathbb{Z}^d}\langle f,\Psi_{j,k}\rangle_{\mathbb{M}_m(\mathbb{R})}\,
\langle \Psi_{j,k},f\rangle_{\mathbb{M}_m(\mathbb{R})}
=\langle f,f\rangle_{\mathbb{M}_m(\mathbb{R})},
\]
and
\[
f=\sum_{(j,k)\in\mathbb{Z}\times\mathbb{Z}^d} C_{j,k}\,\Psi_{j,k},
\]
with convergence in the Hilbert-module norm. In particular, the coefficient family $\{C_{j,k}\}$ uniquely determines $f$.
\end{corollary}

\begin{proof}
The reconstruction formula is exactly the defining identity of a Parseval module frame.
The Parseval-type identity follows by applying left $\mathbb{M}_m(\mathbb{R})$-linearity in the first variable,
passing to finite partial sums in \eqref{eq:parseval-recons}, and using continuity of the inner product
(Proposition~\ref{P1}).
Finally, uniqueness of coefficients follows from Proposition~\ref{prop:coeff-unique}: if all coefficients vanish,
then the reconstruction gives $f=0$.
\end{proof}

The following definition corresponds to the multiresolution decomposition of the Hilbert module $(L^2(\mathbb{R}^d,\mathbb{R}^m),\langle\cdot,\cdot\rangle_{\mathbb{M}_m(\mathbb{R})})$ \cite{wood2004wavelets}. We use a $d\times d$ integer dilation matrix $A$ whose eigenvalues have modulus strictly larger than $1$ \cite{chen2008biorthogonal}.

\begin{definition}[\cite{xia1996vector}]\label{DMRA}
We say that $\Phi$ is a scaling function for a vector-valued multiresolution analysis $\{V_j\}_{j\in \mathbb{Z}}$ of $L^2(\mathbb{R}^d,\mathbb{R}^m)$ if the sequence $\{V_j\}_{j\in \mathbb{Z}}$ satisfies:
\begin{enumerate}[label=\textup{(VMRA\arabic*)}]
\item $V_{j}\subset V_{j+1}$ for all $j\in \mathbb{Z}$.
\item $\bigcap_{j\in\mathbb{Z}} V_j=\{0\}$ and $\bigcup_{j\in \mathbb{Z}} V_j$ is dense in $L^2(\mathbb{R}^d,\mathbb{R}^m)$.
\item $h\in V_j$ if and only if $h(A\cdot)\in V_{j+1}$ for all $j\in\mathbb{Z}$.
\item $h\in V_j$ if and only if $h(\cdot-k)\in V_{j}$ for all $j\in\mathbb{Z}$ and $k\in\mathbb{Z}^d$.
\item The sequence $\{\Phi(\cdot-k)\}_{k\in\mathbb{Z}^d}$ is an orthonormal basis of $V_0$.
\end{enumerate}
\end{definition}

The following proposition, stated in \cite{xia1996vector}, shows that vector-valued wavelets in $L^2(\mathbb{R},\mathbb{R}^m)$ form a broader framework than $m$-multiwavelets in $L^2(\mathbb{R})$, since they can be generated componentwise. An  "if and only if" version is proved in the sequel.

\begin{proposition}[\cite{xia1996vector}]\label{XP2}
Let $\Phi=(\phi_1,\dots,\phi_m)^{\top}$ be a vector-valued scaling function associated with a VMRA of $L^2(\mathbb{R},\mathbb{R}^m)$, and let $\Psi=(\psi_1,\dots,\psi_m)^{\top}$ be its corresponding vector-valued mother wavelet. Then the components $\phi_l$ ($l=1,\dots,m$) constitute $m$ scaling functions, and the components $\psi_l$ ($l=1,\dots,m$) constitute $m$ mother wavelets of $L^2(\mathbb{R})$.
\end{proposition}

\section{From scalar wavelets to multivariate vector-valued wavelets}\label{sec:construct}
To avoid confusion over terminology, let us start by clarifying certain terms used in this article.
\vspace*{-0.4cm}
\paragraph{Multiwavelets:}
In this article, the term \textit{multiwavelets} in $L^2(\mathbb{R}^d,\mathbb{R}^m)$ refers to a componentwise approach to vector-valued wavelets: the components are multivariate scalar scaling functions and mother wavelets in $L^2(\mathbb{R}^d,\mathbb{R})$, and the system is not required to be an orthonormal basis of $L^2(\mathbb{R}^d,\mathbb{R}^m)$. This is the flexible, componentwise framework commonly used in applications.
\vspace*{-0.4cm}
\paragraph{\( m \)-Multiwavelet:}
An \emph{$m$-multiwavelet} in $L^2(\mathbb R)$ consists of $m$ scaling functions $(\phi_1,\dots,\phi_m)$ and $m$ mother wavelets $(\psi_1,\dots,\psi_m)$ whose integer translates and dyadic dilates form an orthonormal basis of $L^2(\mathbb R)$. In our setting, the components of the vector-valued generators $(\Phi,\Psi)$ constructed below provide such an induced $m$-multiwavelet system, which is the input for the tensor-product step in dimension $d\ge2$.

We now describe the constructive route from a scalar wavelet basis to a multivariate vector-valued one.
Before turning to the concrete wavelet construction, we record a structural principle that will be used repeatedly: a Hilbert-orthogonal basis that is simultaneously $\mathcal A$-orthonormal yields a Parseval $\mathcal A$-module frame. This statement will be promoted to a theorem later and will justify the intrinsic reconstruction formulas.

The section is organized in three steps.
In Subsection~\ref{subsec:vv-1d} we first build a vector-valued wavelet system in $L^2(\mathbb{R},\mathbb{R}^m)$ by grouping dyadic scales into blocks of length $m$ (equivalently, by working with the dilation $D=2^m$). This bundling synchronizes the $m$ channels into a single module-valued scale and leads to intrinsic matrix-valued coefficients via the $\mathbb{M}_m(\mathbb{R})$-valued inner product.
In the next subsection, we recall the standard tensor-product construction in dimension $d\ge2$, with the only nuance that the input comes from an induced $m$-multiwavelet system.
Finally, in the last subsection we combine these two ingredients to produce multivariate vector-valued wavelet bases in $L^2(\mathbb{R}^d,\mathbb{R}^m)$ together with the corresponding intrinsic reconstruction.

\begin{lemma}[Equivalence of $\tau$-norm and module norm]\label{lem:tau_module_equiv}
Let $\mathcal A$ be a finite-dimensional $C^*$-algebra and let $\tau$ be a faithful trace on $\mathcal A$.
Then there exist constants $0<c\le C$ such that for all $a\in\mathcal A$ with $a\ge 0$,
\[
c\,\|a\|\le \tau(a)\le C\,\|a\|.
\]
Consequently, for every $x$ in a Hilbert $\mathcal A$-module $(\mathbb H,\langle\cdot,\cdot\rangle_{\mathcal A})$,
\[
c\,\|x\|_{\mathcal A}^2\le \|x\|_{\tau}^2\le C\,\|x\|_{\mathcal A}^2,
\qquad
\|x\|_{\tau}^2:=\tau(\langle x,x\rangle_{\mathcal A}).
\]
\end{lemma}

\begin{proof}
Here, $a\ge 0$ means $a\in\mathcal A_+$, where $\mathcal A_+:=\{b^*b:\,b\in\mathcal A\}$.
Since $\mathcal A$ is finite-dimensional, the set
$K:=\{a\in\mathcal A_+:\|a\|=1\}$ is compact. By faithfulness, $\tau(a)>0$ for all $a\in K$, hence
$c:=\min_{a\in K}\tau(a)>0$ and $C:=\max_{a\in K}\tau(a)<\infty$. Homogeneity yields
$c\|a\|\le\tau(a)\le C\|a\|$ for all $a\ge 0$. The norm equivalence for $\|x\|_{\tau}$ and $\|x\|_{\mathcal A}$ follows by
applying the bounds to $a=\langle x,x\rangle_{\mathcal A}$.
\end{proof}

\begin{theorem}[Hilbert totality and $\mathcal A$-orthonormality $\Rightarrow$ Parseval $\mathcal A$-module frame ]
\label{thm:hilbert_onb_implies_parseval_module}
Let $\mathcal A$ be a finite-dimensional $C^*$-algebra and let $\mathbb H$ be a Hilbert $\mathcal A$-module with inner product
$\langle\cdot,\cdot\rangle_{\mathcal A}$. Fix a faithful trace $\tau$ on $\mathcal A$ and define the scalar inner product
$\langle x,y\rangle_{\tau}=\tau(\langle x,y\rangle_{\mathcal A})$ on $\mathbb H$.
Assume that $\{u_n\}_{n\in\mathbb N}\subset\mathbb H$ satisfies
\begin{itemize}
\item[(i)] $\langle u_n,u_k\rangle_{\mathcal A}=\delta_{n,k}\,1_{\mathcal A}$ for all $n,k$;
\item[(ii)] $\{u_n\}$ is total in the Hilbert space $(\mathbb H,\langle\cdot,\cdot\rangle_{\tau})$.
\end{itemize}
Then $\{u_n\}$ is a Parseval $\mathcal A$-module frame for $\mathbb H$, that is, for every $f\in\mathbb H$,
\begin{equation}\label{eq:parseval_module_recons_thm}
f=\sum_{n\in\mathbb N}\langle f,u_n\rangle_{\mathcal A}\,u_n,
\end{equation}
with convergence in the module norm (equivalently, in the norm induced by $\langle\cdot,\cdot\rangle_{\tau}$).
Moreover, the following Parseval-type identity holds:
\begin{equation}\label{eq:parseval_module_identity_thm}
\sum_{n\in\mathbb N}\langle f,u_n\rangle_{\mathcal A}\,\langle u_n,f\rangle_{\mathcal A}
=\langle f,f\rangle_{\mathcal A}.
\end{equation}
\end{theorem}

\begin{proof}
Fix $f\in\mathbb H$ and write $C_n=\langle f,u_n\rangle_{\mathcal A}\in\mathcal A$.
For $N\ge 1$ set $S_N=\sum_{n=1}^N C_n u_n\in\mathbb H$.
By Lemma~\ref{lem:tau_module_equiv}, $\|\cdot\|_{\tau}$ is equivalent to the module norm, hence $\mathbb H$ is complete for
$\|\cdot\|_{\tau}$.

\smallskip\noindent
\textbf{Step 1: residual identity.}
Using $\mathcal A$-linearity in the first argument, the relation
$\langle x,y\rangle_{\mathcal A}=\langle y,x\rangle_{\mathcal A}^{*}$, and (i), we obtain
\begin{align*}
\langle f-S_N,f-S_N\rangle_{\mathcal A}
&=\langle f,f\rangle_{\mathcal A}
-\sum_{n=1}^N\langle f, C_n u_n\rangle_{\mathcal A}
-\sum_{n=1}^N\langle C_n u_n,f\rangle_{\mathcal A}
+\sum_{n,k=1}^N\langle C_n u_n, C_k u_k\rangle_{\mathcal A} \\
&=\langle f,f\rangle_{\mathcal A}-\sum_{n=1}^N C_n C_n^{*}.
\end{align*}
Here we used the left-module structure to write
\[
\langle f, C_n u_n\rangle_{\mathcal A}
\;=\;\langle C_n u_n,f\rangle_{\mathcal A}^{*}
\;=\;(C_n\langle u_n,f\rangle_{\mathcal A})^{*}
\;=\;C_n C_n^{*},
\]
and
\[
\langle C_n u_n, C_k u_k\rangle_{\mathcal A}
\;=\;C_n\langle u_n, C_k u_k\rangle_{\mathcal A}
\;=\;C_n\,\delta_{n,k}\,C_k^{*},
\]
so the double sum reduces to $\sum_{n=1}^N C_n C_n^{*}$.
In particular, the sequence $\sum_{n=1}^N C_n C_n^{*}$ is increasing in the Loewner order and bounded above by
$\langle f,f\rangle_{\mathcal A}$.

\smallskip\noindent
\textbf{Step 2: convergence of the series and reconstruction.}
For $M<N$, orthonormality gives
\[
\langle S_N-S_M,S_N-S_M\rangle_{\mathcal A}
=\sum_{n=M+1}^N\sum_{k=M+1}^N C_n\,\langle u_n,u_k\rangle_{\mathcal A}\,C_k^{*}
=\sum_{n=M+1}^N C_n C_n^{*},
\]
so $\{S_N\}$ is Cauchy in the module norm because $\sum_{n=1}^N C_n C_n^{*}$ is increasing and bounded above by
$\langle f,f\rangle_{\mathcal A}$. Since $\mathcal A$ is finite-dimensional, this monotone sequence converges in norm, hence
$S_N\to Tf$ for some $Tf\in\mathbb H$.
By continuity of the inner product in the first variable and (i), for each $k$,
\[
\langle f-Tf,u_k\rangle_{\mathcal A}
=\lim_{N\to\infty}\big(\langle f,u_k\rangle_{\mathcal A}-\langle S_N,u_k\rangle_{\mathcal A}\big)
=\lim_{N\to\infty}\big(C_k-\sum_{n=1}^N C_n\delta_{n,k}\big)=0.
\]
Hence $f-Tf$ is $\mathcal A$-orthogonal to $\mathrm{span}\{u_n\}$. By (ii), $\mathrm{span}\{u_n\}$ is dense for the
$\tau$-norm, and by Lemma~\ref{lem:tau_module_equiv} it is also dense for the module norm.
The map $y\mapsto \langle f-Tf,y\rangle_{\mathcal A}$ is continuous (Cauchy--Schwarz), so it vanishes on the closure of
$\mathrm{span}\{u_n\}$ and hence on all of $\mathbb H$. In particular,
\[
\langle f-Tf,f-Tf\rangle_{\mathcal A}=0,
\]
which implies $f-Tf=0$ and yields \eqref{eq:parseval_module_recons_thm}.

\smallskip\noindent
\textbf{Step 3: Parseval identity.}
From Step~1 and $S_N\to f$ in the module norm we get $\langle f-S_N,f-S_N\rangle_{\mathcal A}\to 0$, hence
\[
\sum_{n=1}^N C_n C_n^{*}\xrightarrow[N\to\infty]{}\langle f,f\rangle_{\mathcal A}.
\]
Since $C_n=\langle f,u_n\rangle_{\mathcal A}$ and $\langle u_n,f\rangle_{\mathcal A}=C_n^{*}$, this is
\eqref{eq:parseval_module_identity_thm}.
\end{proof}
\subsection{Vector-valued wavelets in \texorpdfstring{$L^2(\mathbb{R},\mathbb{R}^m)$}{L2(R,Rm)} by grouping dyadic scales}\label{subsec:vv-1d}

Let $\{\phi(\cdot-k)\}_{k\in\mathbb{Z}}\cup\{2^{j/2}\psi(2^j\cdot-k)\}_{j\ge 0,k\in\mathbb{Z}}$ be an orthonormal wavelet basis of $L^2(\mathbb{R})$ associated with an MRA.

Fix $m\ge 2$ and set $D=2^m$. Define the vector-valued scaling function $\Phi:\mathbb{R}\to\mathbb{R}^m$ and the vector-valued mother wavelet $\Psi:\mathbb{R}\to\mathbb{R}^m$ by
\begin{equation}\label{eq:Phi-def}\Phi(x)=\big(\phi(x),\,\psi(x),\,2^{1/2}\psi(2x),\,\dots,\,2^{(m-2)/2}\psi(2^{m-2}x)\big)^{\top},
\end{equation}
\begin{equation}\label{eq:Psi-def}
\Psi(x)=\big(2^{(m-1)/2}\psi(2^{m-1}x),\,2^{m/2}\psi(2^mx),\,\dots,\,2^{(2m-2)/2}\psi(2^{2m-2}x)\big)^{\top}.
\end{equation}
For $k\in\mathbb{Z}$ and $j\ge 0$, set
\[
\Phi_k(x)=\Phi(x-k),\qquad \Psi_{j,k}(x)=D^{j/2}\,\Psi(D^j x-k).
\]

\begin{theorem}[From scalar wavelets to vector-valued wavelets]\label{thm:vv-1d}
Let $(\phi,\psi)$ be an orthonormal wavelet basis of $L^2(\mathbb{R})$ associated with an MRA.
Fix $m\ge 2$ and set $D=2^m$. Define $\Phi,\Psi$ by \eqref{eq:Phi-def}--\eqref{eq:Psi-def}.
For $k\in\mathbb Z$ and $j\ge 0$ set
\[
\widetilde\Phi_k=\Phi(\cdot-k),
\qquad
\widetilde\Psi_{j,k}=D^{j/2}\Psi(D^j\cdot-k).
\]

Then the family
\[
\{\widetilde\Phi_k\}_{k\in\mathbb Z}\ \cup\ \{\widetilde\Psi_{j,k}\}_{j\ge 0,\ k\in\mathbb Z}
\]

\begin{itemize}
\item[(a)] is an orthogonal wavelet basis of the Hilbert space $L^2(\mathbb{R},\mathbb{R}^m)$ (for the usual scalar $L^2$ inner product); moreover $\|\widetilde\Phi_k\|_{L^2}^2=\|\widetilde\Psi_{j,k}\|_{L^2}^2=m$,
\item[(b)] is an $\mathbb{M}_m(\mathbb{R})$--orthonormal system in the Hilbert $\mathbb{M}_m(\mathbb{R})$--module
$\big(L^2(\mathbb{R},\mathbb{R}^m),\langle\cdot,\cdot\rangle_{\mathbb{M}_m(\mathbb{R})}\big)$, i.e.
\[
\big\langle \widetilde U_\alpha,\widetilde U_\beta\big\rangle_{\mathbb{M}_m(\mathbb{R})}
=\delta_{\alpha,\beta}\,I_m
\quad(\alpha,\beta\in\{(0,k)\}\cup\{(j,k):j\ge 0\}),
\]
and it is a Parseval $\mathbb{M}_m(\mathbb{R})$--module frame with the canonical reconstruction
\[
f=\sum_{k\in\mathbb Z}\big\langle f,\widetilde\Phi_k\big\rangle_{\mathbb{M}_m(\mathbb{R})}\,\widetilde\Phi_k
+\sum_{j\ge 0}\sum_{k\in\mathbb Z}\big\langle f,\widetilde\Psi_{j,k}\big\rangle_{\mathbb{M}_m(\mathbb{R})}\,\widetilde\Psi_{j,k},
\qquad f\in L^2(\mathbb{R},\mathbb{R}^m),
\]
with convergence in the module norm.
Moreover, each component inherits compact support, regularity and vanishing moments from $\phi$ and $\psi$.
\end{itemize}
\end{theorem}

\begin{proof}
We identify $L^2(\mathbb R,\mathbb R^m)\simeq \bigoplus_{r=1}^m L^2(\mathbb R)$.
Write the scalar orthonormal wavelet family as
$\{\phi_{k}\}_{k\in\mathbb Z}\cup\{\psi_{n,k}\}_{n\ge 0,k\in\mathbb Z}$ with
$\psi_{n,k}(x)=2^{n/2}\psi(2^n x-k)$.

\smallskip\noindent
\textbf{(1) Orthogonality in the Hilbert sense.}
By construction, each $\Phi(\cdot-k)$ consists of the block
\[
(\phi_k,\psi_{0,k},\psi_{1,k},\dots,\psi_{m-2,k})^T,
\]
and each $D^{j/2}\Psi(D^j\cdot-k)$ consists of the next $m$ dyadic wavelet levels, grouped at the $D$-scale:
\[
D^{j/2}\Psi(D^j\cdot-k)=
(\psi_{jm+m-1,k},\psi_{jm+m,k},\dots,\psi_{jm+2m-2,k})^T.
\]
Using orthonormality across translates and scales in $L^2(\mathbb R)$,
\[
\langle \Phi(\cdot-k),\Phi(\cdot-k')\rangle_{L^2(\mathbb R,\mathbb R^m)} = m\,\delta_{k,k'},
\qquad
\langle D^{j/2}\Psi(D^j\cdot-k),D^{j'/2}\Psi(D^{j'}\cdot-k')\rangle_{L^2(\mathbb R,\mathbb R^m)}
= m\,\delta_{j,j'}\delta_{k,k'}.
\]
Mixed products vanish because the scales in $\Phi(\cdot-k)$ lie in
$V_0\oplus W_0\oplus\cdots\oplus W_{m-2}$ whereas those in $D^{j/2}\Psi(D^j\cdot-k)$ lie in
$\bigoplus_{n\ge m-1} W_n$.
Therefore, the family $\{\widetilde\Phi_k\}\cup\{\widetilde\Psi_{j,k}\}$ is orthogonal in the usual Hilbert space sense, and
\[
\langle \widetilde\Phi_k,\widetilde\Phi_{k'}\rangle_{L^2}=m\,\delta_{k,k'},
\qquad
\langle \widetilde\Psi_{j,k},\widetilde\Psi_{j',k'}\rangle_{L^2}=m\,\delta_{j,j'}\delta_{k,k'}.
\]

\smallskip\noindent
\textbf{(2) Completeness in the Hilbert sense.}
Given $f=(f_1,\dots,f_m)^T\in L^2(\mathbb R,\mathbb R^m)$, expand each component $f_r$ in the scalar wavelet basis,
then regroup the dyadic levels $n\ge 0$ into the disjoint blocks
\[
\{0,1,\dots,m-2\}\quad\text{and}\quad \{jm+m-1,\dots,jm+2m-2\}\ (j\ge 0).
\]
Since distinct wavelet subspaces are orthogonal, the corresponding partial sums associated with disjoint unions of levels
are mutually orthogonal. By the Pythagorean theorem, regrouping terms does not change the $L^2$-limit of the series, and it
therefore yields an expansion of $f$ in the stated vector generators.
Hence the family is an orthogonal basis of $L^2(\mathbb R,\mathbb R^m)$ (with $\|\widetilde\Phi_k\|_{L^2}^2=\|\widetilde\Psi_{j,k}\|_{L^2}^2=m$).
In particular, it is total in $L^2(\mathbb R,\mathbb R^m)$.

\smallskip\noindent
\textbf{(3) $\mathbb{M}_m(\mathbb{R})$--orthonormality and module-Parseval reconstruction.}
Recall the $\mathbb{M}_m(\mathbb{R})$--valued inner product
\[
\langle F,G\rangle_{\mathbb{M}_m(\mathbb{R})}=\int_{\mathbb R}F(x)\,G(x)^{\top}\,dx,
\qquad F,G\in L^2(\mathbb R,\mathbb R^m).
\]
Let $\widetilde U,\widetilde V$ be two generators among $\{\widetilde\Phi_k\}\cup\{\widetilde\Psi_{j,k}\}$.
Using the block structure above, one checks entrywise that the scalar orthogonality relations imply
\[
\big\langle \widetilde U,\widetilde V\big\rangle_{\mathbb{M}_m(\mathbb{R})}
=
\begin{cases}
I_m,& \widetilde U=\widetilde V,\\
0_m,& \widetilde U\neq \widetilde V,
\end{cases}
\]
i.e. the family is $\mathbb{M}_m(\mathbb{R})$--orthonormal.
Together with the Hilbert-space completeness from (2), we can apply
Theorem~\ref{thm:hilbert_onb_implies_parseval_module} to conclude that
$\{\widetilde\Phi_k\}\cup\{\widetilde\Psi_{j,k}\}$ is a Parseval
$\mathbb{M}_m(\mathbb{R})$--module frame and that the canonical reconstruction
formula in (b) holds, with convergence in the module norm, after fixing an enumeration of the double-indexed family as a single sequence $(u_n)$.

\smallskip\noindent
\textbf{(4) VMRA and inherited properties.}
Define $V_0^{\mathrm{vec}}=\overline{\mathrm{span}\{\Phi(\cdot-k)\}}$ and
$V_j^{\mathrm{vec}}=\{g:\ g(D^{-j}\cdot)\in V_0^{\mathrm{vec}}\}$.
The VMRA axioms follow from the scalar MRA and the fact that the wavelet blocks are arranged to match dilation $D$.
Componentwise support/regularity/vanishing moments are inherited from $\phi$ and $\psi$ since each component is a translate/dilate of $\phi$ or $\psi$.
\end{proof}
\begin{remark}\label{rem:vv-multiwavelet}
We emphasize that both multiwavelets and module-valued (vector-valued) wavelets are built from scalar wavelet data.
The module framework used here allows coefficient extraction and reconstruction to be carried out
intrinsically via the $\mathbb{M}_m(\mathbb{R})$-valued inner product on $L^2(\mathbb{R},\mathbb{R}^m)$,
producing matrix-valued coefficients that encode inter-channel correlations and the vector geometry.
\end{remark}
\begin{corollary}\label{cor}
Consider the following classes of wavelet bases:
\begin{enumerate}
    \item A wavelet basis of \(L^2(\mathbb{R})\) associated with an MRA.
    \item A vector-valued wavelet basis of \(L^2(\mathbb{R}, \mathbb{R}^m)\) (\(m \ge 2\)) associated with a VMRA.
    \item An \(m\)-multiwavelet basis of \(L^2(\mathbb{R})\) associated with an MRA.
\end{enumerate}
Then
\[
  \textup{(1)} \;\implies\; \textup{(2)}
  \quad\text{and}\quad
  \textup{(2)} \;\iff\; \textup{(3)}.
\]
In particular, any basis in \textup{(1)} yields bases in \textup{(2)} and \textup{(3)}, and the construction preserves regularity, vanishing moments, and compact support.
\end{corollary}

\begin{proof}
(1) \(\Rightarrow\) (2) is Theorem~\ref{thm:vv-1d}. (2) \(\Rightarrow\) (3) follows from Proposition~\ref{XP2} in \cite{xia1996vector}. Moreover, the resulting \(m\)-multiwavelet respects the MRA in the sense that
\[
V_j=\overline{\operatorname{span}\left\{\phi_i\!\left(2^{j}x-k\right):\, i\in\{1,\dots,m\},\ k\in\mathbb{Z}\right\}}
\]
form a nested sequence of subspaces of \(L^2(\mathbb{R})\).

For (3) \(\Rightarrow\) (2), let \((\phi_1,\dots,\phi_m)\) be the
\(m\)-scaling functions and \((\psi_1,\dots,\psi_m)\) the corresponding
\(m\)-mother wavelets of the multiwavelet MRA in \(L^2(\mathbb{R})\).
Define the vector-valued generators
\[
\Phi=(\phi_1,\dots,\phi_m)^T,\qquad \Psi=(\psi_1,\dots,\psi_m)^T,
\]
and set
\[
\Phi_k(x)=\Phi(x-k),\qquad \Psi_{j,k}(x)=2^{j/2}\Psi(2^j x-k).
\]
Via the componentwise identification
\(L^2(\mathbb{R},\mathbb{R}^m)\simeq \bigoplus_{r=1}^m L^2(\mathbb{R})\),
the orthonormality of \(\{\phi_i(\cdot-k)\}\cup\{2^{j/2}\psi_i(2^j\cdot-k)\}\)
for each \(i\) implies orthonormality of
\(\{\Phi_k\}_{k\in\mathbb{Z}}\cup\{\Psi_{j,k}\}_{j\ge0,k\in\mathbb{Z}}\), and
completeness follows from the completeness of the scalar expansions of each
component. Define \(V_0^{\mathrm{vec}}=\overline{\mathrm{span}\{\Phi_k\}}\) and
\(V_j^{\mathrm{vec}}=\{g:\ g(2^{-j}\cdot)\in V_0^{\mathrm{vec}}\}\). Then the VMRA axioms are inherited from the multiwavelet MRA of \((\phi_i,\psi_i)\).
Thus the vector-valued system is a VMRA wavelet basis, proving (3)
\(\Rightarrow\) (2).
\end{proof}

\begin{remark}\label{rem:m-multiwavelet-importance}
It is well known that $m$-multiwavelets are an important tool in $L^2(\mathbb{R})$.
Within our construction, for each fixed $m$ the vector-valued module method yields an $m$-multiwavelet
system that inherits regularity, compact support, and vanishing moments from the underlying scalar basis.
Classical multiwavelet constructions typically proceed via refinement equations, and achieving prescribed
properties for a given multiplicity can be technically involved.
\end{remark}
\subsection{Multivariate scalar wavelets in \texorpdfstring{$L^2(\mathbb{R}^d)$}{L2(Rd)} from m-multiwavelets }\label{subsec:multi-tensor}

We briefly recall the standard tensor-product construction that lifts an $m$-multiwavelet in $L^2(\mathbb{R})$ to a multivariate scalar wavelet system in $L^2(\mathbb{R}^d)$.
Let $\{\phi_i\}_{i=1}^m$ and $\{\psi_i\}_{i=1}^m$ be an orthonormal $m$-multiwavelet system in $L^2(\mathbb{R})$
with associated MRA spaces $\{V_j\}_{j\in\mathbb{Z}}$ and wavelet spaces $\{W_j\}_{j\in\mathbb{Z}}$, so that
$V_{j+1}=V_j\oplus W_j$ and $\overline{\bigcup_{j\in\mathbb{Z}}V_j}=L^2(\mathbb{R})$.

For $d\ge 2$, define the tensor-product spaces
\[
\mathbf{V}_j=V_j^{\otimes d}\subset L^2(\mathbb{R}^d).
\]

\begin{proposition}\label{prop:tensor-mra-multi}
The family $\{\mathbf{V}_j\}_{j\in\mathbb{Z}}$ is an MRA of $L^2(\mathbb{R}^d)$ with dyadic dilation $2$ and integer
translations. Moreover, for each $j\in\mathbb{Z}$ one has the orthogonal decomposition
\[
\mathbf{V}_{j+1}
=
\mathbf{V}_j\ \oplus\ \bigoplus_{\epsilon\in\{0,1\}^d\setminus\{0\}} \mathbf{W}_j^{(\epsilon)},
\]
where
\[
\mathbf{W}_j^{(\epsilon)}
=
\bigotimes_{\ell=1}^d X_{j}^{(\epsilon_\ell)},
\qquad
X_{j}^{(0)}=V_j,\ \ X_{j}^{(1)}=W_j.
\]
In particular,
\[
L^2(\mathbb{R}^d)
=
\mathbf{V}_0\ \oplus\ \bigoplus_{j\ge 0}\ \bigoplus_{\epsilon\in\{0,1\}^d\setminus\{0\}} \mathbf{W}_j^{(\epsilon)},
\]
and the associated orthonormal wavelet basis consists of the usual $2^d-1$ wavelet types (one for each
$\epsilon\neq 0$), obtained by tensoring the scalar scaling functions and mother wavelets componentwise.
\end{proposition}

\begin{proof}
Since $V_{j+1}=V_j\oplus W_j$ orthogonally in $L^2(\mathbb{R})$, taking $d$-fold tensor products yields
\[
\mathbf{V}_{j+1}
=
(V_j\oplus W_j)^{\otimes d}
=
\bigoplus_{\epsilon\in\{0,1\}^d}\ \bigotimes_{\ell=1}^d X_{j}^{(\epsilon_\ell)},
\]
an orthogonal direct sum because tensor products preserve orthogonality. The term corresponding to
$\epsilon=0$ is exactly $\mathbf{V}_j=V_j^{\otimes d}$; collecting the remaining $2^d-1$ terms gives the stated
decomposition.
For the MRA axioms: nestedness and translation invariance are immediate from $\mathbf{V}_j=V_j^{\otimes d}$
and the corresponding properties of $V_j$. Scaling holds because $f\in V_j$ iff $f(2\cdot)\in V_{j+1}$, and this
equivalence persists under tensor products, so $F\in \mathbf{V}_j$ iff $F(2\cdot)\in \mathbf{V}_{j+1}$. For density,
finite linear combinations of pure tensors from $\bigcup_j V_j$ are dense in $L^2(\mathbb{R}^d)$, and each such tensor
lies in some $\mathbf{V}_J$ (take $J$ larger than all indices used in its factors). For the trivial intersection,
let $F\in \bigcap_j \mathbf{V}_j$. Then for any $j$ and any $G\in \mathbf{V}_j^\perp$, we have $\langle F,G\rangle=0$.
Since $\overline{\bigcup_j \mathbf{V}_j}=L^2(\mathbb{R}^d)$, it follows that $\bigcap_j \mathbf{V}_j=\{0\}$.
Finally, summing the orthogonal decompositions over $j\ge 0$ gives the stated expansion of $L^2(\mathbb{R}^d)$.
\end{proof}
\subsection{Multivariate vector-valued wavelets in \texorpdfstring{$L^2(\mathbb{R}^d,\mathbb{R}^m)$}{L2(Rd,Rm)}}

We now combine Theorem~\ref{thm:vv-1d} with the tensor-product MRA construction in
Proposition~\ref{prop:tensor-mra-multi}. Starting from a scalar orthonormal wavelet basis $(\phi,\psi)$ of
$L^2(\mathbb{R})$, Theorem~\ref{thm:vv-1d} yields a $\mathbb{M}_m(\mathbb{R})$-orthonormal vector-valued wavelet
system $(\Phi,\Psi)$ in $L^2(\mathbb{R},\mathbb{R}^m)$ and (by Remark~\ref{rem:vv-multiwavelet}) an induced
$m$-multiwavelet MRA $\{V_j\}$ in $L^2(\mathbb{R})$. For $d=1$ this already gives the desired vector-valued
construction. Hence assume $d\ge 2$, set $\mathbf V_j=V_j^{\otimes d}\subset L^2(\mathbb{R}^d)$ as defined in
Subsection~\ref{subsec:multi-tensor}, and let
$\{g_{j,k}^{\epsilon,\alpha}\}$ be the scalar tensor-product wavelet basis provided by
Proposition~\ref{prop:tensor-mra-multi}, with $\epsilon\in\{0,1\}^d\setminus\{0\}$, $\alpha\in\{1,\dots,m\}^d$,
$k\in\mathbb{Z}^d$, and $j\ge 0$ (together with the scaling generators at $j=0$). We then group $m$ scalar
generators into $\mathbb{R}^m$-valued ones.

\begin{theorem}\label{thm:main-vv-multi}
From any scalar orthonormal wavelet basis of $L^2(\mathbb{R})$, one can construct a
$\mathbb{M}_m(\mathbb{R})$-orthonormal multivariate vector-valued wavelet basis of
$L^2(\mathbb{R}^d,\mathbb{R}^m)$. Moreover, the expansion coefficients are matrix-valued and are obtained through
the module inner product $\langle \cdot,\cdot\rangle_{\mathbb{M}_m(\mathbb{R})}$.
\end{theorem}

\begin{proof}
By Theorem~\ref{thm:vv-1d}, the vector-valued system $(\Phi,\Psi)$ is $\mathbb{M}_m(\mathbb{R})$-orthonormal,
and its components generate an $m$-multiwavelet MRA $\{V_j\}$ in $L^2(\mathbb{R})$.
For $d\ge 2$, Proposition~\ref{prop:tensor-mra-multi} provides a scalar tensor-product wavelet basis
$\{g_{j,k}^{\epsilon,\alpha}\}$ of $L^2(\mathbb{R}^d)$ indexed by $(\epsilon,\alpha,j,k)$ as above.
Fix any bijection
\[
\{1,\dots,m\}^d \ \longleftrightarrow\ \{1,\dots,m^{d-1}\}\times\{1,\dots,m\},
\qquad
\alpha \longmapsto (l,r).
\]
For each multivariate scalar generator $g_{j,k}^{\epsilon,\alpha}$, form an $\mathbb{R}^m$-valued generator by grouping
the $m$ terms corresponding to a fixed $l$:
\[
G_{j,k}^{\epsilon,l}
=
\big(g_{j,k}^{\epsilon,(l,1)},\dots,g_{j,k}^{\epsilon,(l,m)}\big)^{T}\in L^2(\mathbb{R}^d,\mathbb{R}^m),
\]
and similarly at scale $j=0$ for the scaling generators. Orthonormality in the module sense follows entrywise:
for any two generators $G,G'$ from this family,
\[
\langle G,G'\rangle_{\mathbb{M}_m(\mathbb{R})}
=
\delta_{G,G'}\,I_m,
\]
because the entries are scalar inner products of the orthonormal basis $\{g_{j,k}^{\epsilon,\alpha}\}$.

Completeness follows because each component $f_r$ of any $f=(f_1,\dots,f_m)^T\in L^2(\mathbb{R}^d,\mathbb{R}^m)$
admits an expansion in the scalar multivariate basis. Grouping these $m$ scalar expansions yields an expansion of $f$ in
$\{G_{j,k}^{\epsilon,l}\}$.

Finally, the matrix-valued coefficients are recovered by the module inner product,
$C_{j,k}^{\epsilon,l}=\langle f, G_{j,k}^{\epsilon,l}\rangle_{\mathbb{M}_m(\mathbb{R})}$, and the reconstruction and
Parseval identities follow from Theorem~\ref{thm:hilbert_onb_implies_parseval_module}.
Concerning the properties of the constructed vector-valued wavelets, they are inherited from the initial scalar
wavelet in the same way as in the standard separable setting: compact support, regularity, and vanishing moments
propagate componentwise and through tensor products. Moreover, the construction preserves the VMRA conditions, since the
induced vector-valued spaces are obtained by combining the VMRA in $d=1$ with the tensor-product MRA in
$L^2(\mathbb{R}^d)$.
\end{proof}

\section{Examples in \texorpdfstring{$L^2(\mathbb{R}^2,\mathbb{R}^2)$}{L2(R2,R2)}}\label{sec:examples}
\noindent We present the case $d=m=2$ to keep the construction transparent. The next corollary
specializes Theorem~\ref{thm:main-vv-multi} to this setting, and we illustrate it with the Haar and
Daubechies $(\mathrm{db}4)$ bases.
\begin{corollary}\label{P}
Let $(\phi_1,\phi_2)$ and $(\psi_1,\psi_2)$ be an orthonormal $2$-multiwavelet system in $L^2(\mathbb{R})$
associated with an MRA, and set
$\phi=(\phi_1,\phi_2)^T$, $\psi=(\psi_1,\psi_2)^T$.
Then the vector-valued scaling functions $\{\Phi^i_k\}_{i\in\{1,2\},k\in\mathbb{Z}^2}$ and wavelet functions
$\{\Psi^i_{j,k}\}_{i\in\{1,\dots,6\},k\in\mathbb{Z}^2,\ j\ge 0}$ defined below form an orthogonal basis of
$L^2(\mathbb{R}^2,\mathbb{R}^2)$, where
\begin{equation*}
     \begin{cases}
        &\Phi^1_k(x,y)=(\phi_1(x-k_1)\phi_1(y-k_2),\phi_2(x-k_1)\phi_2(y-k_2))^T \\
        & \Phi^2_k(x,y)=(\phi_1(x-k_1)\phi_2(y-k_2),\phi_2(x-k_1)\phi_1(y-k_2))^T,
     \end{cases}
\end{equation*}
\begin{equation*}
     \begin{cases}
        &\Psi^1_{j,k}(x,y)=(\phi_1(2^jx-k_1)\psi_1(2^jy-k_2),\phi_2(2^jx-k_1)\psi_2(2^jy-k_2))^T \\
        & \Psi^2_{j,k}(x,y)=(\phi_1(2^jx-k_1)\psi_2(2^jy-k_2),\phi_2(2^jx-k_1)\psi_1(2^jy-k_2))^T\\
        &\Psi^3_{j,k}(x,y)=(\psi_1(2^jx-k_1)\phi_1(2^jy-k_2),\psi_2(2^jx-k_1)\phi_2(2^jy-k_2))^T \\
        & \Psi^4_{j,k}(x,y)=(\psi_1(2^jx-k_1)\phi_2(2^jy-k_2),\psi_2(2^jx-k_1)\phi_1(2^jy-k_2))^T,
\end{cases}
\end{equation*}
\begin{equation*}
\begin{cases}
   &\Psi^5_{j,k}(x,y)=(\psi_1(2^{j}x-k_1)\psi_1(2^{j}y-k_2),\psi_2(2^{j}x-k_1)\psi_2(2^{j}y-k_2))^T \\
   & \Psi^6_{j,k}(x,y)=(\psi_1(2^{j}x-k_1)\psi_2(2^{j}y-k_2),\psi_2(2^{j}x-k_1)\psi_1(2^{j}y-k_2))^T.
\end{cases}
\end{equation*}

\end{corollary}

\begin{proof}
Let $\{g_{j,k}^{\epsilon,\alpha}\}$ be the scalar tensor-product basis of $L^2(\mathbb{R}^2)$ from
Proposition~\ref{prop:tensor-mra-multi}, applied to the MRA associated with
$(\phi_1,\phi_2)$ and $(\psi_1,\psi_2)$, where $\epsilon\in\{0,1\}^2$ and $\alpha\in\{1,2\}^2$.
Choose the grouping
\[
A_1=\{(1,1),(2,2)\},\qquad A_2=\{(1,2),(2,1)\}.
\]
For $\epsilon=(0,0)$ (scaling) and $\epsilon\neq(0,0)$ (wavelets), define the vector-valued generators by
\[
\Phi_k^l=\big(g_{0,k}^{(0,0),\alpha_{l,1}},\,g_{0,k}^{(0,0),\alpha_{l,2}}\big)^T,\qquad
\Psi_{j,k}^{\epsilon,l}=\big(g_{j,k}^{\epsilon,\alpha_{l,1}},\,g_{j,k}^{\epsilon,\alpha_{l,2}}\big)^T,
\]
where $\alpha_{l,1},\alpha_{l,2}$ are the two elements of $A_l$.
Writing $g_{j,k}^{\epsilon,\alpha}$ explicitly gives exactly the formulas displayed in the statement.
Orthogonality and completeness follow from the scalar orthonormality of $\{g_{j,k}^{\epsilon,\alpha}\}$ and the
componentwise identification $L^2(\mathbb{R}^2,\mathbb{R}^2)\simeq L^2(\mathbb{R}^2)\oplus L^2(\mathbb{R}^2)$.
Equivalently, this is Theorem~\ref{thm:main-vv-multi} specialized to $d=m=2$ with the above grouping.
\end{proof}

\begin{example}
The Haar basis is the classical orthonormal wavelet system in $L^2(\mathbb{R})$. By Theorem~\ref{thm:vv-1d}
and Corollary~\ref{cor} (with $m=2$), it induces an orthonormal $2$-multiwavelet system associated with an MRA,
so Corollary~\ref{P} yields a vector-valued wavelet basis in $L^2(\mathbb{R}^2,\mathbb{R}^2)$. The Haar scaling
function $\phi$ is the characteristic function of $[0,1]$, and the mother wavelet is
$\psi(x)=\phi(2x)-\phi(2x-1)$. With $m=2$ the dilation is $D=4$ (grouping two dyadic levels), and
Theorem~\ref{thm:vv-1d} gives the vector-valued generators

\begin{equation*}
    \Phi(\cdot)=(\phi(\cdot),\psi(\cdot))^T,\;\Psi(\cdot)=(2^{1/2}\psi(2\cdot),2\psi(4\cdot))^T
\end{equation*}

\begin{figure}[H] 
  \begin{center}
  \includegraphics[scale=0.4]{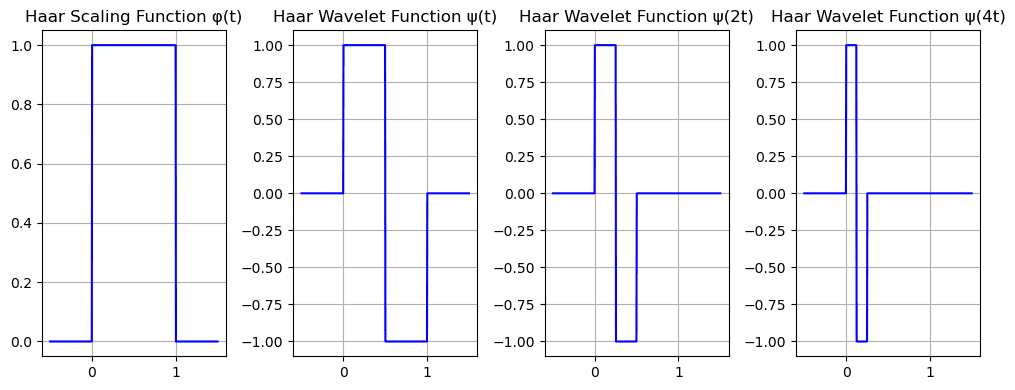}
\caption{The Haar scaling function and the Haar wavelet function in three scales $j=0,1,2$. }
    \label{fig:1}
 \end{center}
 \end{figure}

\begin{figure}[H]
  \centering
  \begin{tabular}{cc}
    \includegraphics[width=0.45\linewidth]{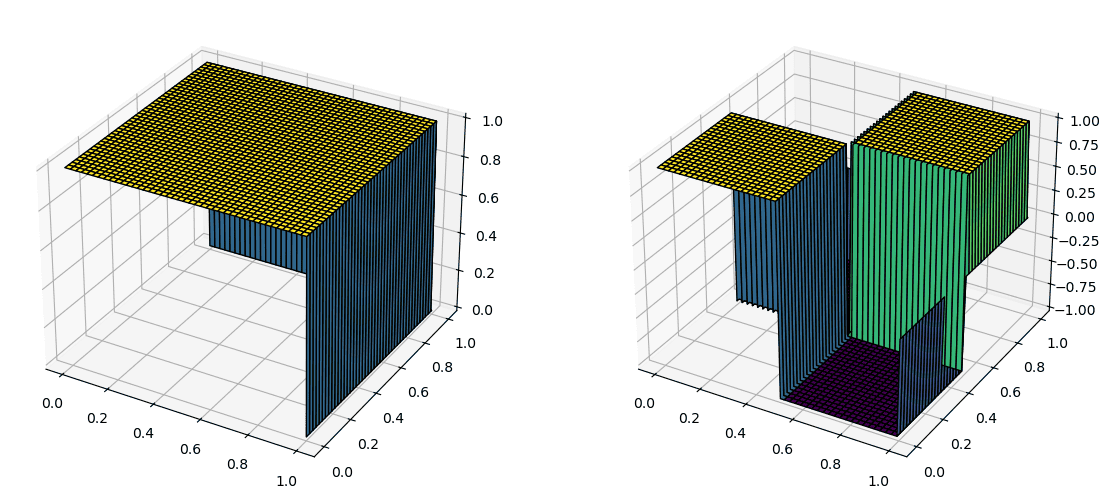} &
    \includegraphics[width=0.45\linewidth]{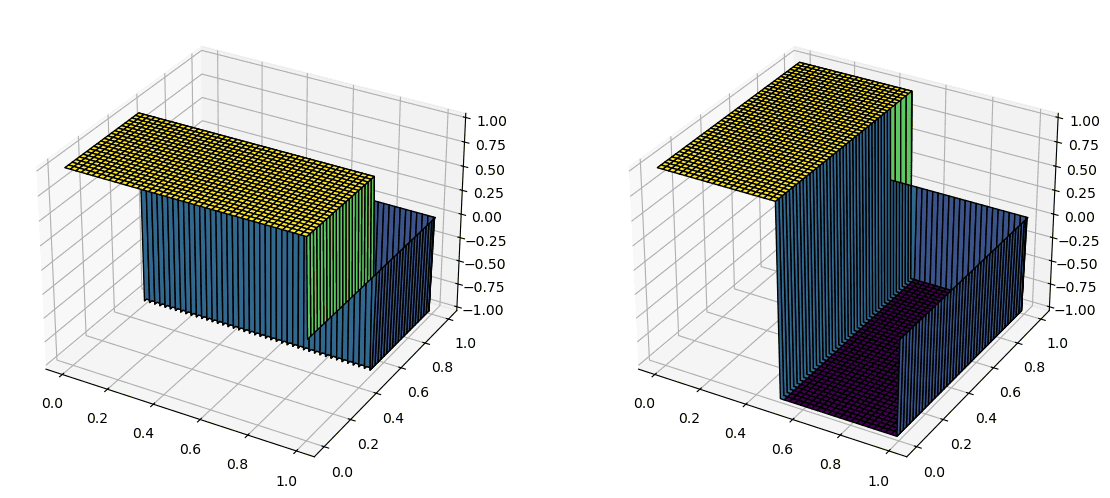} \\
  \end{tabular}
  \caption{Haar vector-valued scaling components $\Phi^1$ (left) and $\Phi^2$ (right).}
  \label{fig:haar-phi-components}
\end{figure}

\begin{figure}[H]
  \centering
  \begin{tabular}{ccc}
    \includegraphics[width=0.30\linewidth]{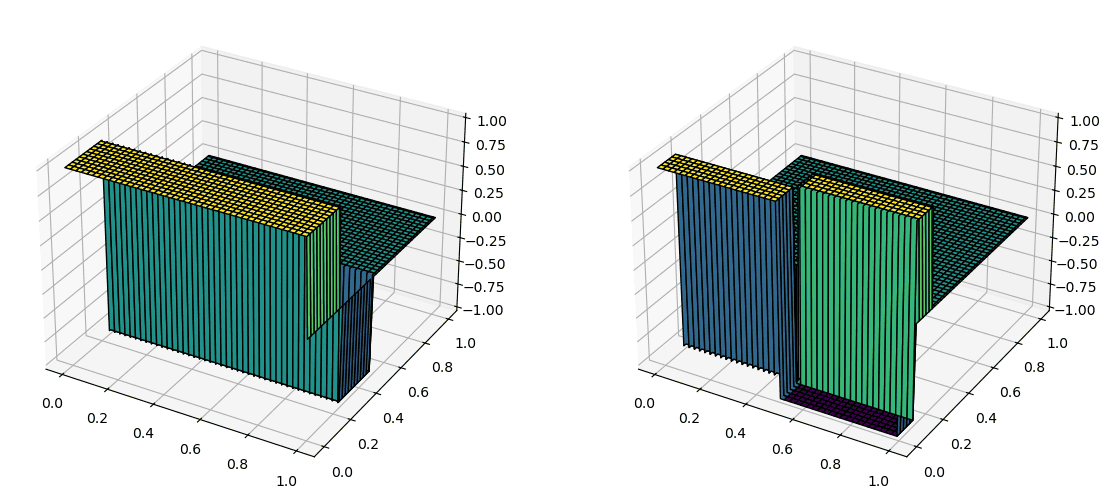} &
    \includegraphics[width=0.30\linewidth]{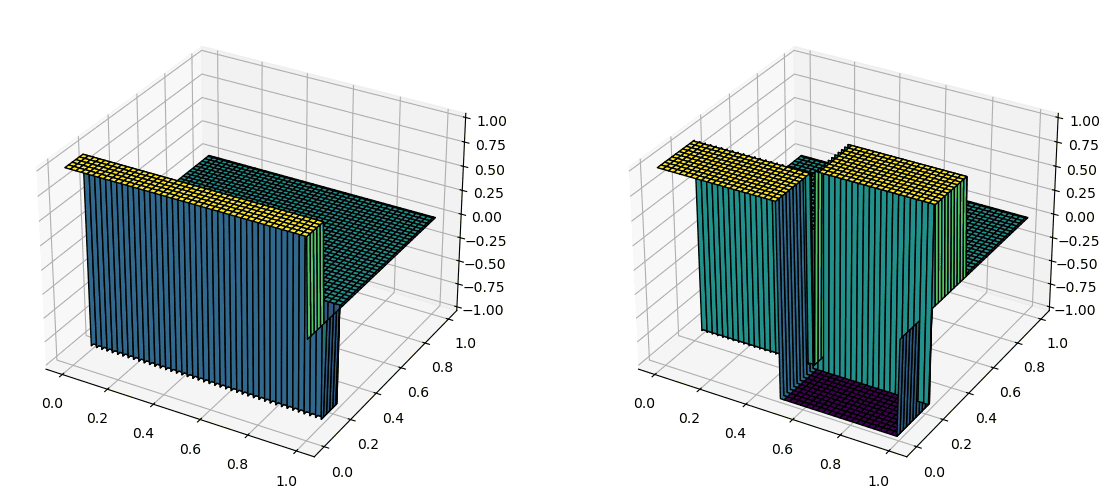} &
    \includegraphics[width=0.30\linewidth]{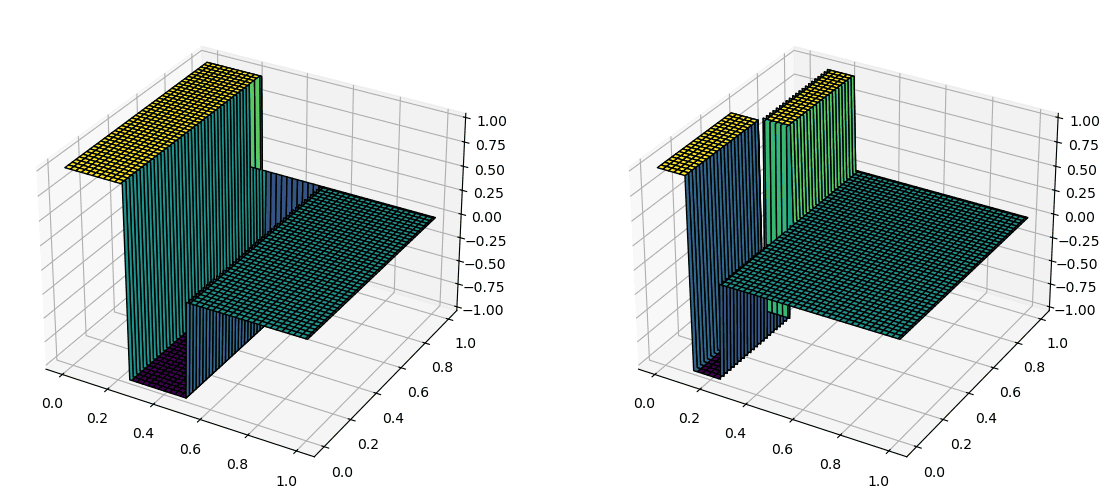} \\
    \includegraphics[width=0.30\linewidth]{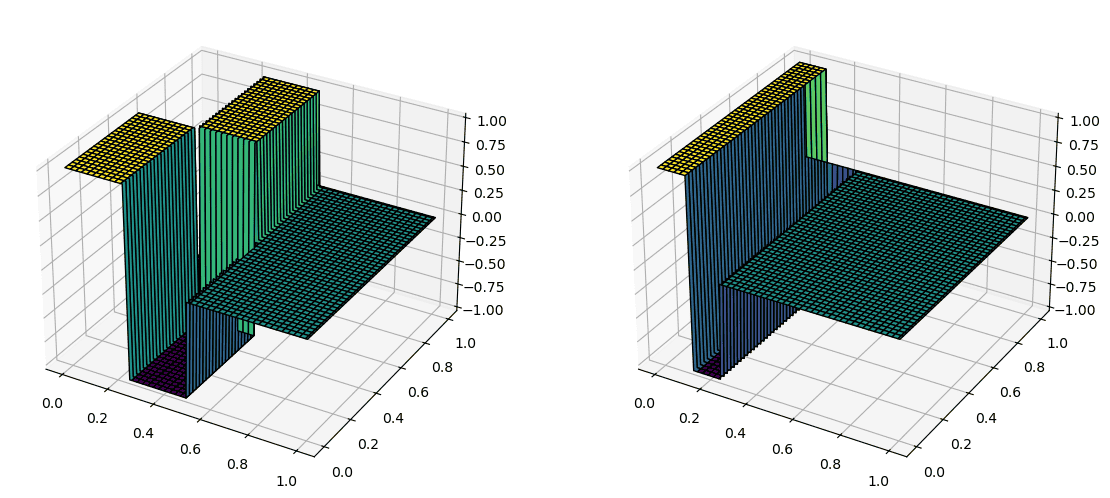} &
    \includegraphics[width=0.30\linewidth]{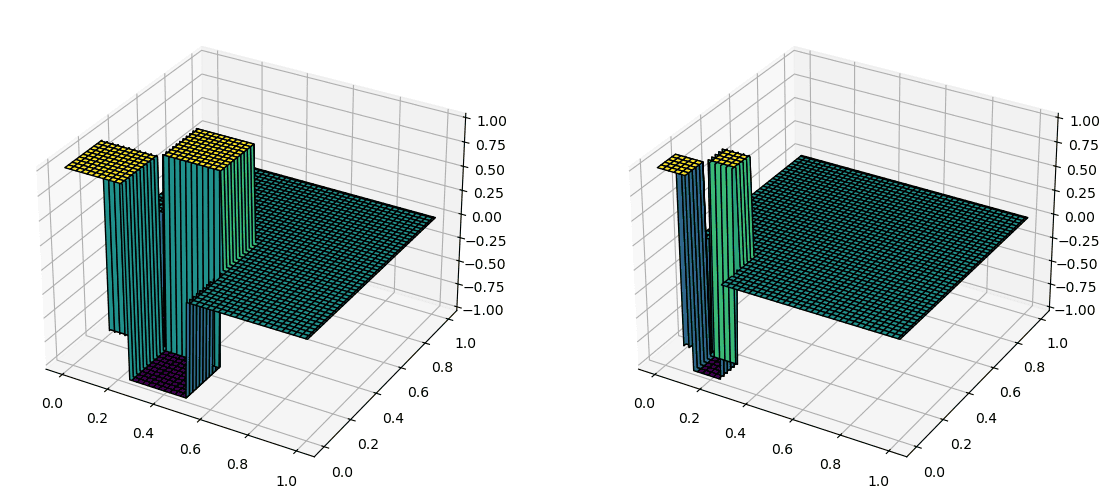} &
    \includegraphics[width=0.30\linewidth]{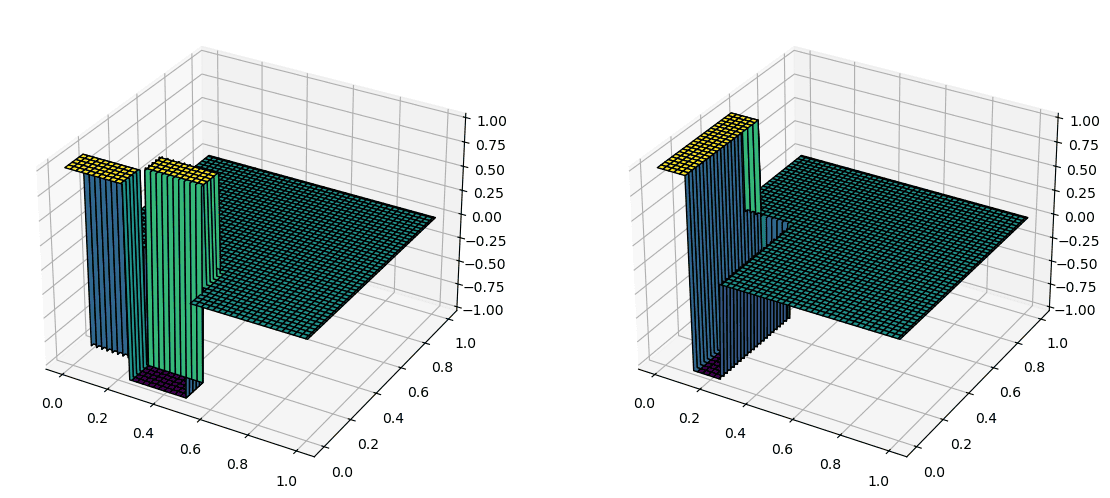} \\
  \end{tabular}
  \caption{Haar vector-valued wavelet components $\Psi^1$--$\Psi^6$.}
  \label{fig:haar-psi-components}
\end{figure}

Figure~\ref{fig:1} shows the scalar Haar scaling function and wavelet at scales $j=0,1,2$;
Figures~\ref{fig:haar-phi-components} and~\ref{fig:haar-psi-components} show the corresponding vector-valued components.

\end{example}

\begin{example}
The Daubechies \(\mathrm{db}4\) wavelet basis is a compactly supported orthonormal system in \(L^2(\mathbb{R})\)
with higher-order smoothness and vanishing moments compared to the Haar basis. By Theorem~\ref{thm:vv-1d} and
Corollary~\ref{cor} (with $m=2$), it induces an orthonormal $2$-multiwavelet system associated with an MRA, so
Corollary~\ref{P} yields a vector-valued wavelet basis in \(L^2(\mathbb{R}^2,\mathbb{R}^2)\). The scaling function
\(\phi\) is defined by its refinement equation with low-pass coefficients \((h_0,h_1,h_2,h_3)\), and the mother
wavelet \(\psi\) is generated by the corresponding high-pass filter \((g_0,g_1,g_2,g_3)\); see
\cite{daubechies1992ten}. With $m=2$ the dilation is $D=4$ (grouping two dyadic levels), and
Theorem~\ref{thm:vv-1d} gives the vector-valued generators

\begin{equation*}
    \Phi(\cdot)=(\phi(\cdot),\psi(\cdot))^T,\;\Psi(\cdot)=(2^{1/2}\psi(2\cdot),2\psi(4\cdot))^T.
\end{equation*}

\begin{figure}[H] 
    \begin{center}
    \includegraphics[scale=0.4]{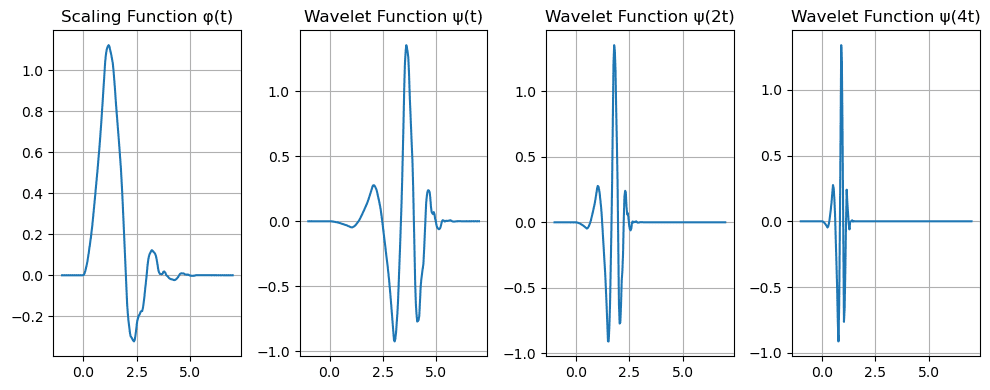}
    \caption{The Daubechies scaling function and the Daubechies wavelet function in three scales $j=0,1,2$. }
    \label{fig:10}
    \end{center}
\end{figure}

\begin{figure}[H]
  \centering
  \begin{tabular}{cc}
    \includegraphics[width=0.45\linewidth]{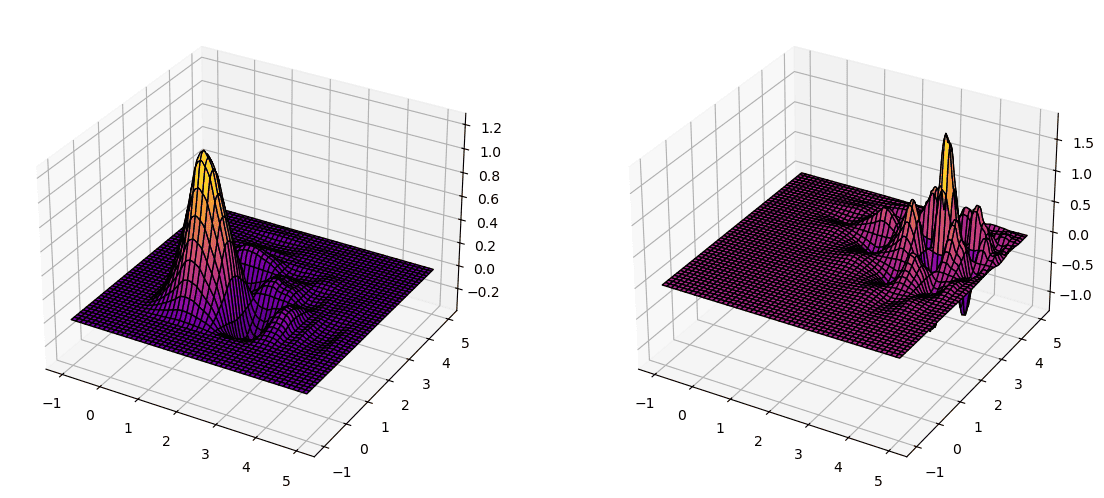} &
    \includegraphics[width=0.45\linewidth]{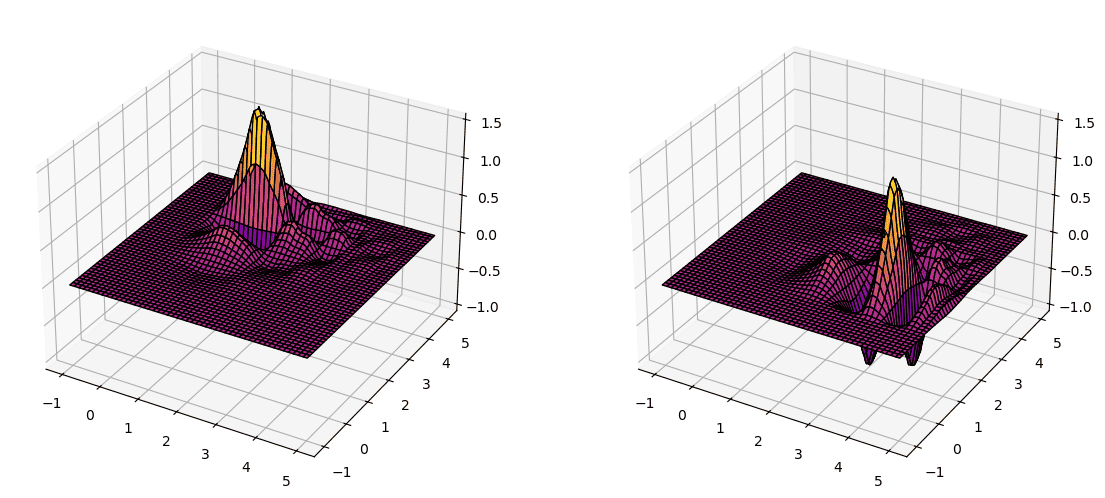} \\
  \end{tabular}
  \caption{Daubechies $(\mathrm{db}4)$ vector-valued scaling components $\Phi^1$ (left) and $\Phi^2$ (right).}
  \label{fig:db4-phi-components}
\end{figure}

\begin{figure}[H]
  \centering
  \begin{tabular}{ccc}
    \includegraphics[width=0.30\linewidth]{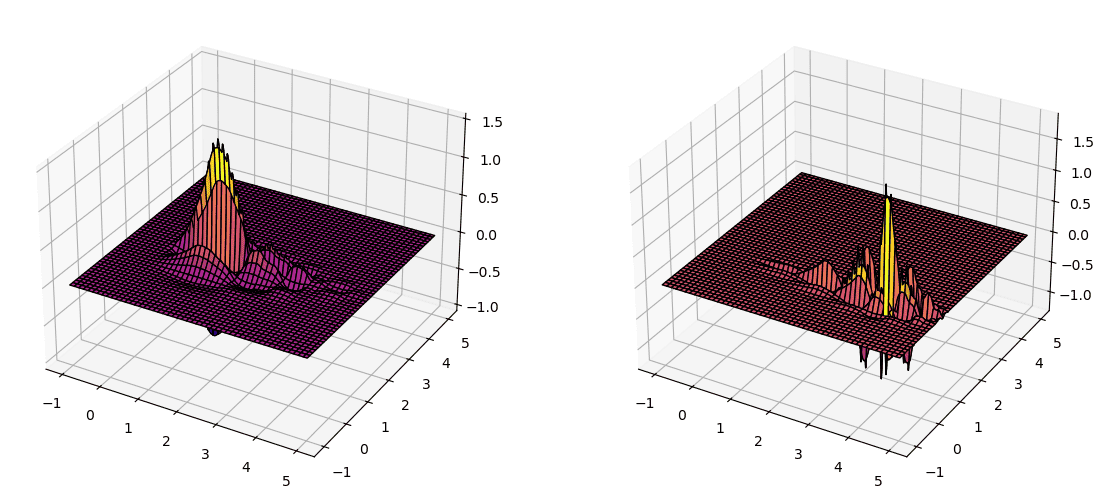} &
    \includegraphics[width=0.30\linewidth]{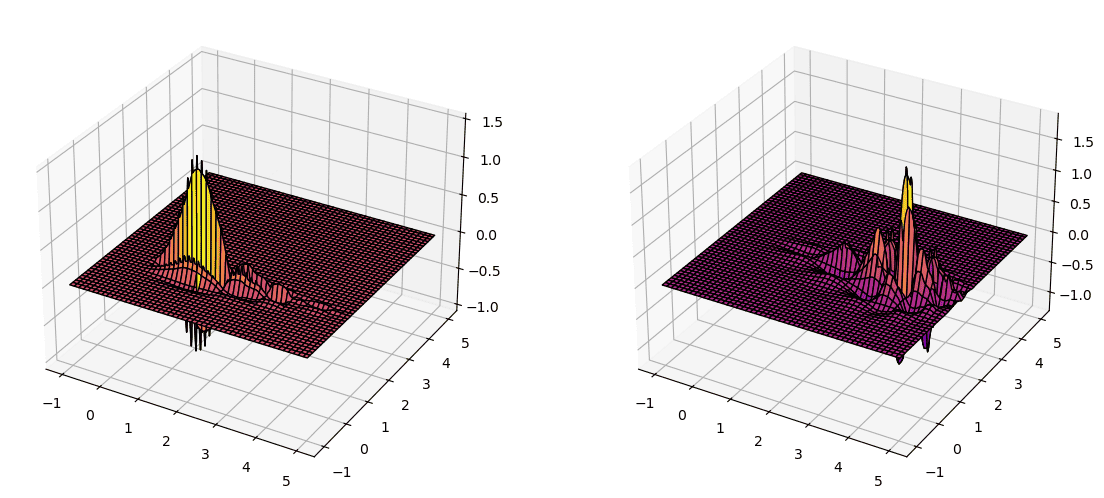} &
    \includegraphics[width=0.30\linewidth]{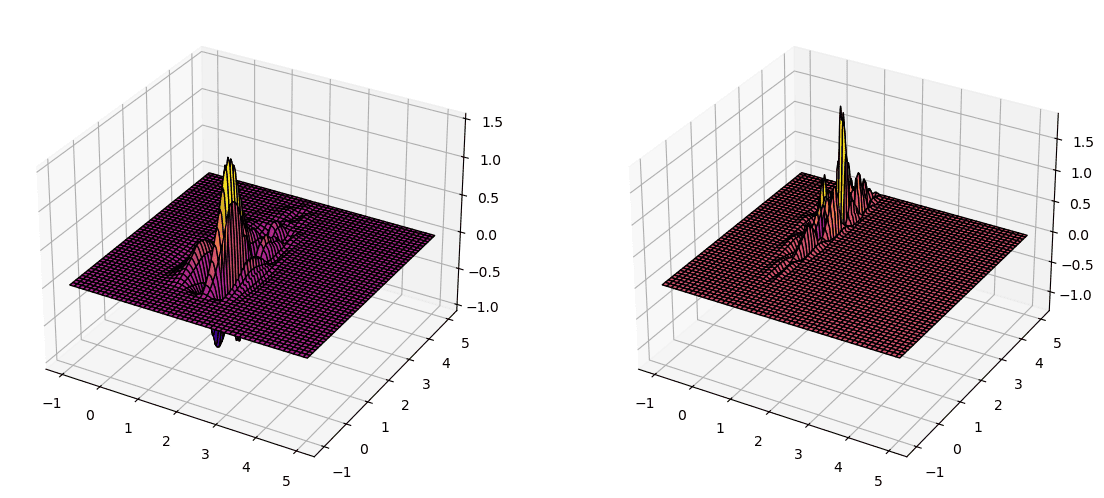} \\
    \includegraphics[width=0.30\linewidth]{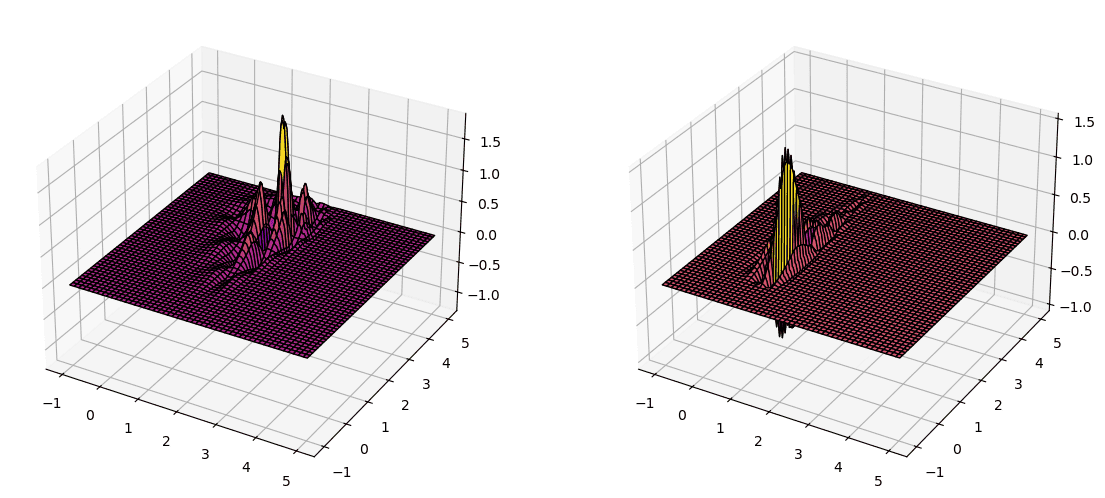} &
    \includegraphics[width=0.30\linewidth]{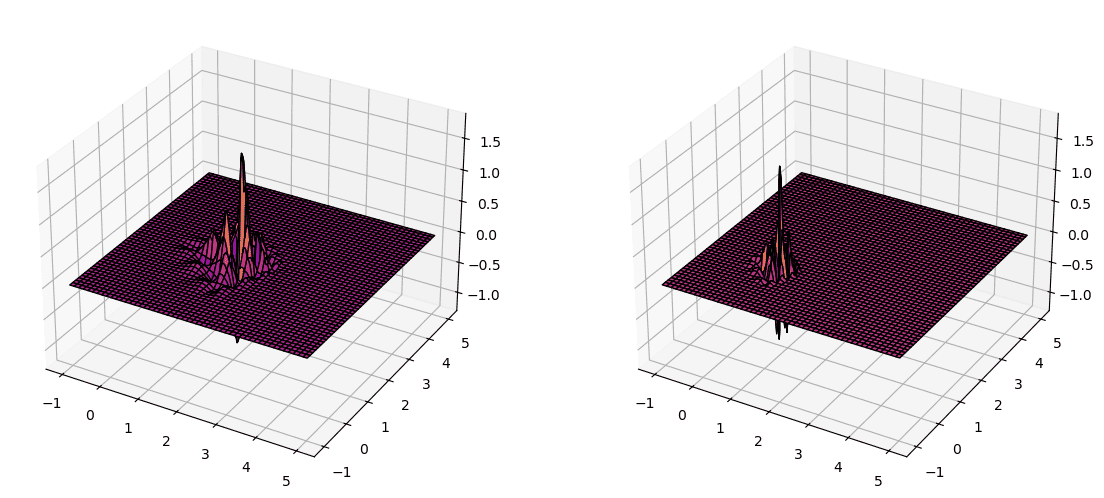} &
    \includegraphics[width=0.30\linewidth]{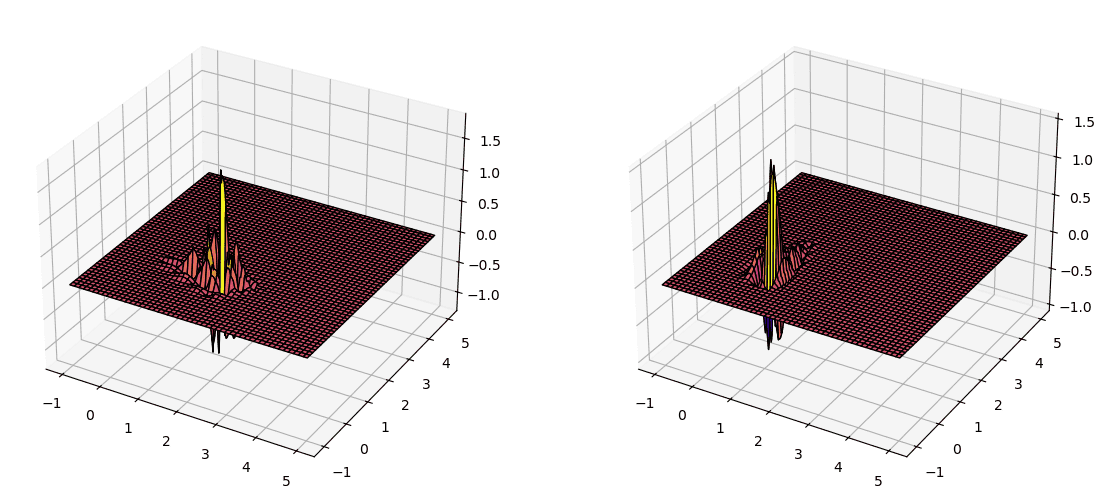} \\
  \end{tabular}
  \caption{Daubechies $(\mathrm{db}4)$ vector-valued wavelet components $\Psi^1$--$\Psi^6$.}
  \label{fig:db4-psi-components}
\end{figure}

Figure~\ref{fig:10} shows the scalar Daubechies scaling function and wavelet at scales $j=0,1,2$;
Figures~\ref{fig:db4-phi-components} and~\ref{fig:db4-psi-components} show the corresponding vector-valued components.

\end{example}
\section{Conclusion and outlook}
We developed a Hilbert-$\mathbb{M}_m(\mathbb{R})$-module framework for vector-valued wavelets in
$L^2(\mathbb{R}^d,\mathbb{R}^m)$ and used it to construct orthonormal bases directly from scalar wavelet
systems. The module inner product yields matrix-valued coefficients with intrinsic reconstruction and Parseval
identities, while the lifting procedure preserves compact support, smoothness, and vanishing moments. The resulting
vector-valued bases retain the structural information of multichannel data that is lost under purely componentwise
scalar expansions, and they avoid solving refinement equations explicitly.

Several directions merit further study. On the analytic side, the approach extends naturally to other coefficient
algebras and to frames, with potential benefits for numerical stability and redundancy control. On the constructive
side, nonseparable and anisotropic designs, as well as data-adaptive choices of groupings, could broaden the class
of vector-valued systems obtainable within this module viewpoint. Finally, numerical implementations and application
driven evaluations in imaging, geophysics, and multichannel signal processing should clarify how the matrix-valued coefficients improve feature extraction, denoising, and classification compared to scalar methods.

\section*{Declarations}

\section*{\small
 Conflict of interest} 

{\small
 The authors declare that they have no conflict of interest.}

\section*{\small
 Data availability} 

 {\small
 No data are associated with this article.}

\section*{\small
 Funding statement} 

 {\small
 This research received no external funding.}





\begin{thebibliography} {99}
\normalsize 


\bibitem{ahmad2022nonuniformsuper}Ahmad, O., AH, A. A., \& Ahmad, M. (2022). NONUNIFORM SUPER WAVELETS IN L2 (K). Problemy Analiza-Issues of Analysis, 11(1), 3-19.



\bibitem{balan1999density}Balan, R. Density and redundancy of the noncoherent Weyl-Heisenberg superframes. {\em Contemporary Mathematics}. \textbf{247} pp. 29-42 (1999)






\bibitem{bhatt2007orthogonal}Bhatt, G., Johnson, B. \& Weber, E. Orthogonal wavelet frames and vector-valued wavelet transforms. {\em Applied And Computational Harmonic Analysis}. \textbf{23}, 215-234 (2007)



\bibitem{chen2007study}Chen, Q. \& Cheng, Z. A study on compactly supported orthogonal vector-valued wavelets and wavelet packets. {\em Chaos, Solitons \& Fractals}. \textbf{31}, 1024-1034 (2007)

\bibitem{chen2008biorthogonal}Chen, Q. \& Shi, Z. Biorthogonal multiple vector-valued multivariate wavelet packets associated with a dilation matrix. {\em Chaos, Solitons \& Fractals}. \textbf{35}, 323-332 (2008)

\bibitem{chen2008biorthogonality}Chen, Q. \& Wei, Z. Biorthogonality property of vector-valued multivariate wavelet packets. {\em 2008 Fourth International Conference On Natural Computation}. pp. 471-475 (2008)





\bibitem{chui1996study}Chui, C. \& Lian, J. A study of orthonormal multi-wavelets. {\em Applied Numerical Mathematics}. \textbf{20}, 273-298 (1996)

\bibitem{cotronei2002multiwavelet}Cotronei, M., Montefusco, L. \& Puccio, L. Multiwavelet analysis and signal processing. {\em IEEE Transactions On Circuits And Systems II: Analog And Digital Signal Processing}. \textbf{45}, 970-987 (2002)


\bibitem{daubechies1988orthonormal}Daubechies, I. Orthonormal bases of compactly supported wavelets. {\em Communications On Pure And Applied Mathematics}. \textbf{41}, 909-996 (1988)


\bibitem{daubechies1992ten}Daubechies, I. Ten lectures on wavelets. (SIAM,1992)

\bibitem{dutkay2005mra}Dutkay, D., Bildea, S. \& Picioroaga, G. MRA superwavelets. {\em New York Journal Of Mathematics}. \textbf{11} pp. 1-19 (2005)

\bibitem{fowler2002wavelet}Fowler, J. \& Hua, L. Wavelet transforms for vector fields using omnidirectionally balanced multiwavelets. {\em IEEE Transactions On Signal Processing}. \textbf{50}, 3018-3027 (2002)
\bibitem{FranckLarson2002}Franck, M. \& Larson, D. Frames in Hilbert $C^*$-modules and $C^*$-algebras. {\em Journal Of Operator Theory}. (2002)






\bibitem{JakobsenLuef2020}Jakobsen, M. \& Luef, F. Duality of Gabor frames and Heisenberg modules. {\em Journal Of Noncommutative Geometry}. \textbf{14}, 1445-1500 (2020)


\bibitem{lance1995hilbert}Lance, E. Hilbert C*-modules: a toolkit for operator algebraists. (Cambridge University Press,1995)

\bibitem{lemarie1986ondelettes}Lemarié, P. \& Meyer, Y. Ondelettes et bases hilbertiennes. {\em Rev. Mat. Iberoamericana}. \textbf{2}, 1-18 (1986)



\bibitem{mallat1989multiresolution}Mallat, S. Multiresolution approximations and wavelet orthonormal bases of $L^2(\mathbb{R})$. {\em Transactions Of The American Mathematical Society}. \textbf{315}, 69-87 (1989)


\bibitem{mallat1999wavelet}Mallat, S. A wavelet tour of signal processing. (Elsevier,1999)

\bibitem{manuilov2000hilbert}Manuilov, V. \& Troitsky, E. Hilbert C*-and W*-modules and their morphisms. {\em Journal Of Mathematical Sciences}. \textbf{98}, 137-201 (2000)


\bibitem{strang1996wavelets}Strang, G. \& Nguyen, T. Wavelets and filter banks. (SIAM,1996)

\bibitem{shukla2018super}Shukla, Niraj K., and Saurabh Chandra Maury. "Super‐wavelets on local fields of positive characteristic." Mathematische Nachrichten 291.4 (2018): 704-719.

\bibitem{wood2004wavelets}Wood, P. Wavelets and Hilbert modules. {\em Journal Of Fourier Analysis And Applications}. \textbf{10}, 573-598 (2004)



\bibitem{xia1996vector}Xia, X. \& Suter, B. Vector-valued wavelets and vector filter banks. {\em IEEE Transactions On Signal Processing}. \textbf{44}, 508-518 (1996)



\end{thebibliography}
\end{document}